\documentclass[a4paper,english,cleveref, autoref, thm-restate, nolineno]{socg-lipics-v2021}

\usepackage{amsmath}
\usepackage{amsbsy}
\usepackage{amssymb}
\usepackage{cite}

\usepackage{multicol}

\usepackage{graphicx}
\usepackage{tabularx}

\usepackage[table,xcdraw]{xcolor}

\usepackage{mathtools}
\usepackage{algorithm}

\usepackage[noend]{algpseudocode}

\usepackage{amscd}

\newtheorem{obs}{Observation}

\DeclareMathOperator*{\argmax}{arg\,max}

\algrenewcommand{\algorithmiccomment}[1]{\hfill// #1}

\newcommand{\myparagraph}[1]{\smallskip\noindent\textbf{#1}}

\newcommand{\new}[1]{\textcolor{black}{\textbf{{#1}}}}

\newcommand{\figb}[1]{\textcolor{black}{\textbf{{#1}}}}

\newcommand{\myskip}[1]{}
\newcommand{\mywidth}{0.8}

\renewcommand{\to}{\rightarrow} 

\newcolumntype{P}[1]{>{\centering\arraybackslash}p{#1}}

\newsavebox\CBox

\usepackage{booktabs}
\usepackage{multirow}
\usepackage{graphicx}

\newcommand{\hw}[1]{}
\newcommand{\mw}[1]{}
\newcommand{\pb}[1]{}
\newcommand{\na}[1]{}

\DeclareMathOperator{\im}{im}
\DeclareMathOperator{\VR}{VR}

\title{Mixup Barcodes: Quantifying Geometric-Topological Interactions between Point Clouds}

\titlerunning{Mixup barcodes}

\author{Hubert Wagner}{University of Florida, Gainesville, US}{hwagner@ufl.edu}{https://orcid.org/0009-0009-9111-8429
}{}
\author{Nickolas Arustamyan}{University of Central Florida, Orlando, US}{nickolas.arustamyan@ucf.edu}{https://orcid.org/0009-0004-2850-5390}{}
\author{Matthew Wheeler}{University of Florida, Gainesville, US}
{matthew.wheeler@medicine.ufl.edu}{https://orcid.org/0000-0001-7396-5469
}{}
\author{Peter Bubenik}{University of Florida, Gainesville, US}{peter.bubenik@ufl.edu}{https://orcid.org/0000-0001-5262-2133}{}

\authorrunning{H. Wagner, N. Arustamyan, M. Wheeler and P. Bubenik} 

\Copyright{Hubert Wagner, Nickolas Arustamyan, Matthew Wheeler, Peter Bubenik} 



\relatedversion{}\relatedversiondetails{Full Version}{https://doi.org/10.48550/arXiv.2402.15058}


\category{} 

\supplement{}
\supplementdetails[subcategory={}, cite={}, swhid={swh:1:dir:9b75232380440dc7556b698dd1b7d07882a61fac}]{Software}{https://github.com/hubwag/Mixup-SoCG26}

\funding{H.W. was partially supported by the 2022 Google Research Scholar Award in Algorithms and Optimization. M.W. was supported by The American Association of Immunologists through an Intersect Fellowship for Computational Scientists and Immunologists.}

\acknowledgements{The authors thank Ond\v{r}ej Draganov, Teresa Heiss, Herbert Edelsbrunner and Yasu Hiraoka for helpful discussions.}

\EventEditors{}
\EventNoEds{2}
\EventLongTitle{The 42nd International Symposium on Computational Geometry (SoCG 2026)}
\EventShortTitle{SoCG 2026}
\EventAcronym{SoCG}
\EventYear{2026}
\EventLogo{figs/socg-logo}
\SeriesVolume{224}
\ArticleNo{63}

\begin{document}
\maketitle

\begin{abstract}

We propose a novel geometric-topological descriptor called a mixup barcode. Intuitively, it characterizes the shape of a point cloud as well as its spatial relationship with another point cloud embedded in the same ambient space. More technically, it enriches a standard persistence barcode with information on the image persistent homology. In three dimensions it captures natural spatial relationships like overlap and surrounding; in higher dimensions more intricate spatial relationships are captured. We provide a theoretical setup and a simple algorithm for mixup barcodes.

As a proof of concept, we explore data arising in a geometric-topological problem from machine learning. Specifically, we take first steps towards verifying a hypothesis stating that geometric-topological relationships within intermediate point cloud representations in an artificial neural network can hinder its training. More broadly, our experiments suggest that mixup barcodes are useful for characterizing spatial relationships and spatial interactions (i.e. the evolution of  spatial relationships) that are hard to directly visualize or capture using standard methods.
\end{abstract}


\section{Introduction}\label{sec:mot}
Computational geometry delivers a variety of tools for characterizing 
the shape of a point cloud. One can compute the convex hull, the Delaunay triangulation and its variants, and related filtrations such as alpha filtrations.  Analogs for high dimensional point clouds such as Rips and \v{C}ech filtrations are becoming better understood~\cite{edelsbrunner2024maximum} and more computable~\cite{ripser, vcufar2020ripserer}.

\begin{figure}
    \centering
    \includegraphics[trim=0cm 0.0cm 0.0cm 0.0cm, clip, width=0.7\linewidth]{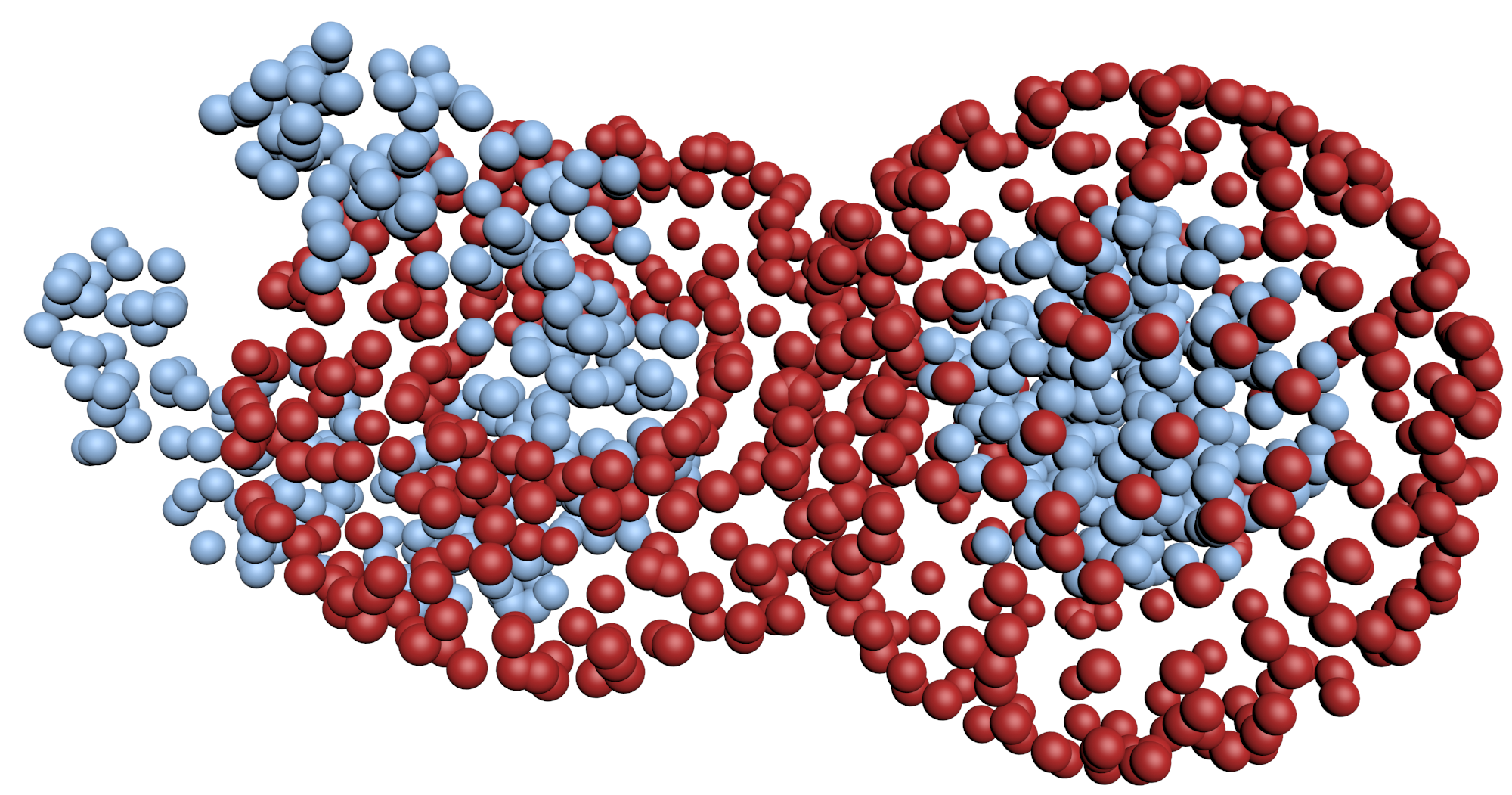}    
    \caption{TDA allows us to characterize the shapes of the two, red and blue, point clouds. We also ask: what are  the spatial relationships between these point clouds?
    Intuitively, a portion of the blue point cloud is surrounded by the spherical part of the red one; similarly, another portion
    of the blue point cloud loops around the toroidal subset of the red points.    
    These are natural questions and we currently lack the tools (and in some cases the language) to discuss them. 
    High-dimensional generalizations play an important role in the machine learning problem we consider. 
    } 
    \label{fig:ex3d_intro}
\end{figure}

\myparagraph{The shape of data.}
In recent years, it has become popular to summarize the geometric-topological structure of a point cloud  using persistent homology. It involves computing a persistence diagram, or barcode, serving as a geometric-topological descriptor of a dataset. 

\myparagraph{Beyond the shape of data.}
There are important problems which require
characterizing the \new{spatial relationships} among datasets.
We provide a new tool, called a \new{mixup barcode}, which 
captures both the shape of data and its spatial relationships with
another dataset embedded in the same ambient space. See~\cref{fig:ex3d_intro} for an example. We will also consider spatial interactions, by which we mean the evolution of spatial relationships over time. 

\begin{figure}
    \centering
\includegraphics[width=\mywidth\textwidth]{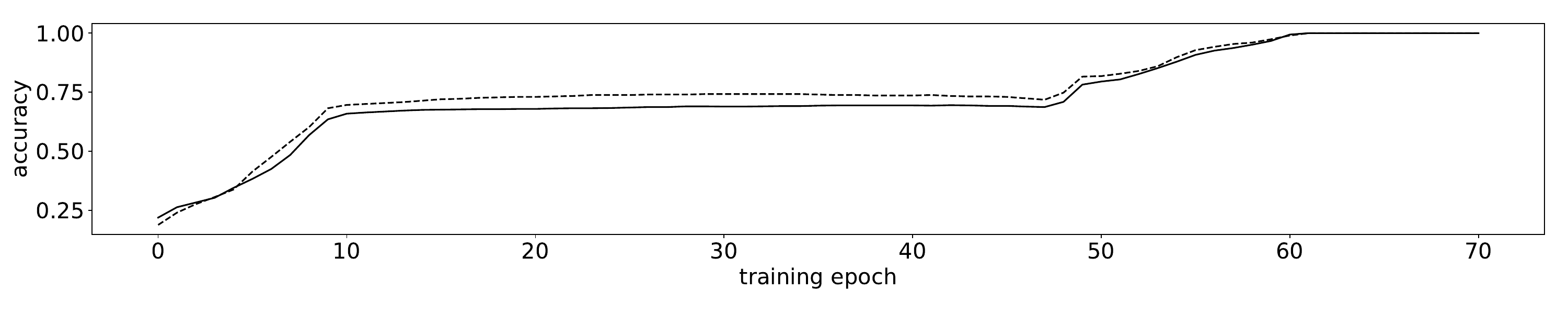}
\includegraphics[width=\mywidth\textwidth]{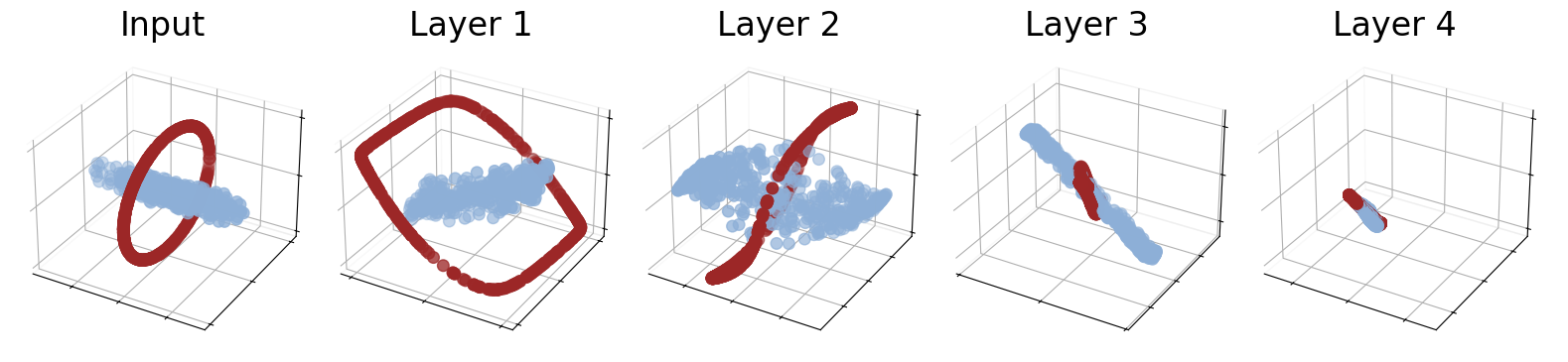}
\includegraphics[trim=0cm 0.0cm 0.0cm 1.2cm, clip, width=\mywidth\textwidth]{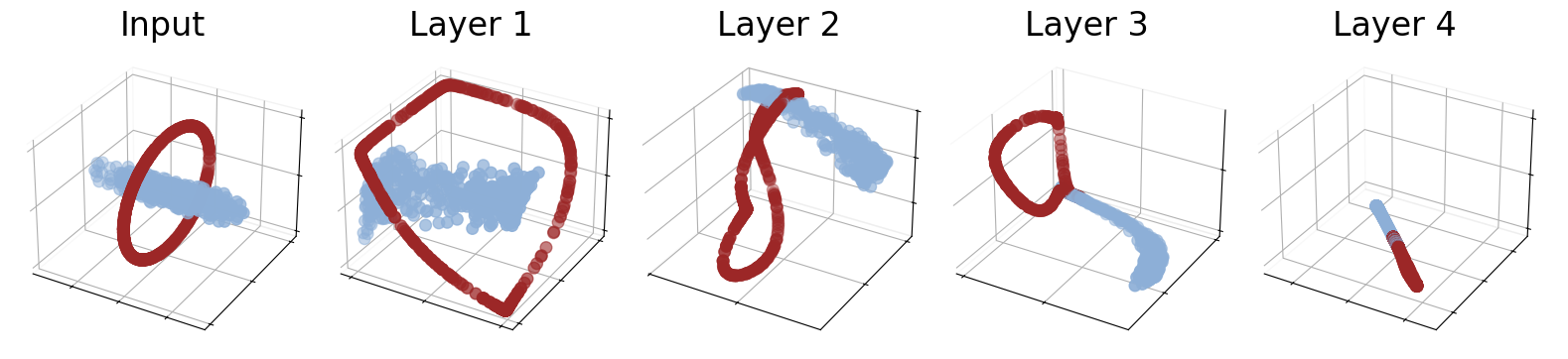}
\includegraphics[trim=0cm 0.0cm 0.0cm 1.2cm, clip, width=\mywidth\textwidth]{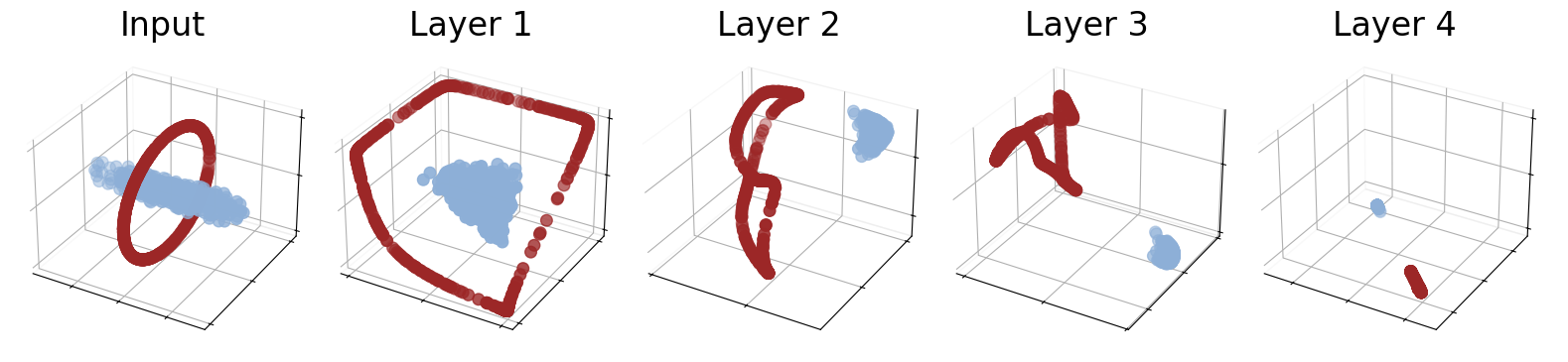} 

\includegraphics[width=\mywidth\textwidth]{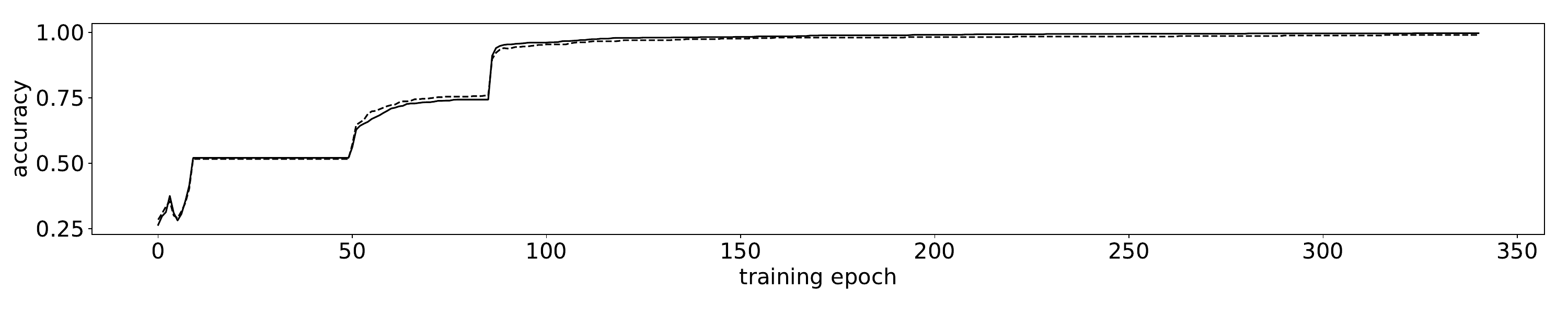}
\includegraphics[trim=0cm 0.0cm 0.0cm 1.2cm, clip, width=\mywidth\textwidth]{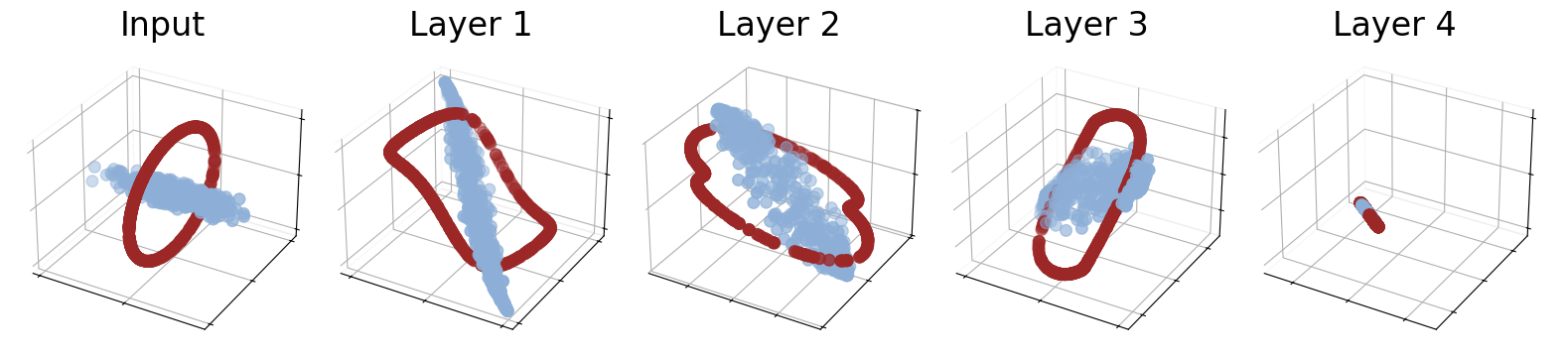}
\includegraphics[trim=0cm 0.0cm 0.0cm 1.2cm, clip, width=\mywidth\textwidth]{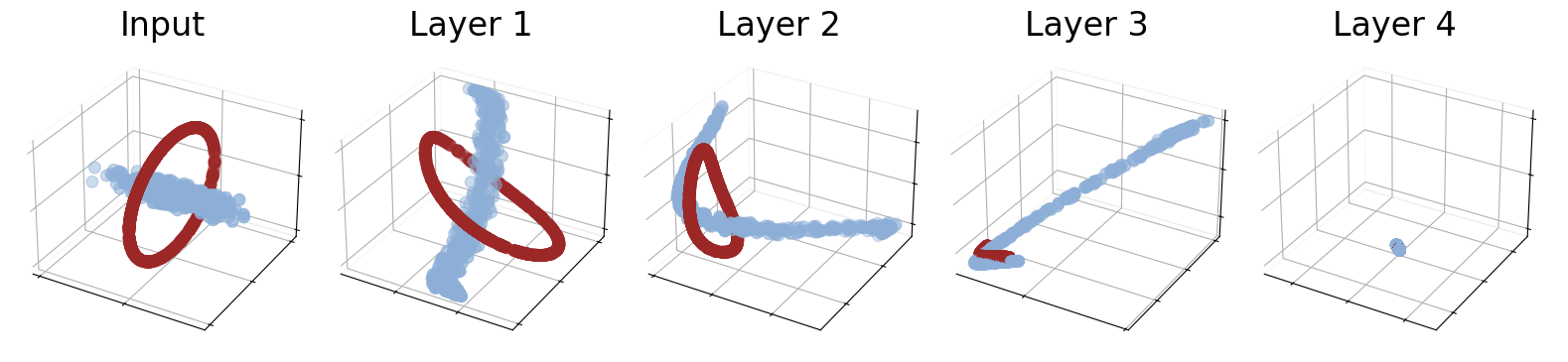}
\includegraphics[trim=0cm 0.0cm 0.0cm 1.2cm, clip, width=\mywidth\textwidth]{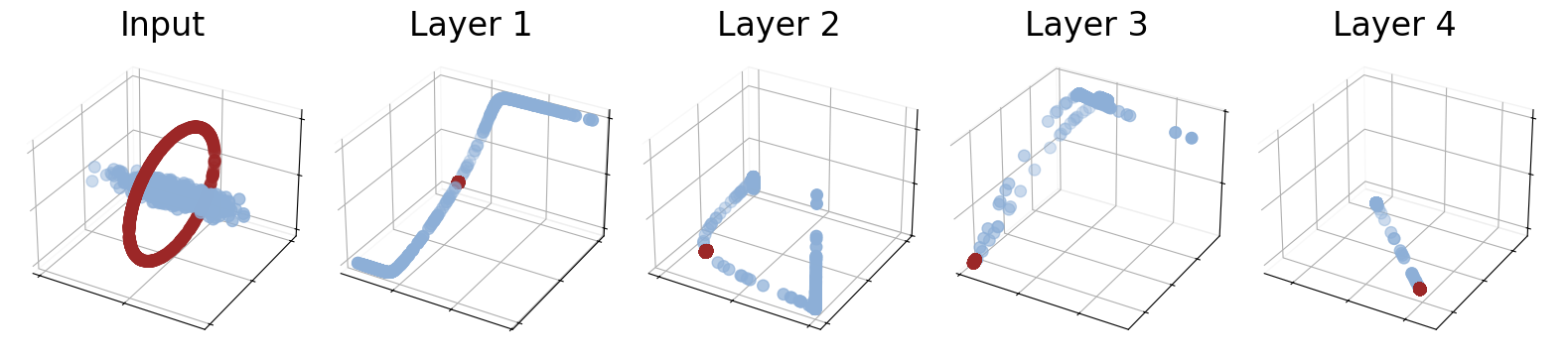}

\caption{Two training runs of the same MLP classifier. For each run, three selected training epochs are shown as rows. Each row shows the outputs of consecutive layers from left to right. Intuitively, at the end of the training, the blue and red points output by the \emph{last} layer should be separated.
In both cases (almost) perfect test accuracy is achieved. However, in the second run the model failed to \emph{disentangle} the red and blue points, and a small portion of the blue points became \emph{confused} with the red points.
In particular, this behavior lowers the generalization power of the model (since certain new inputs are bound to be misclassified). We propose tools for detecting such situations.
}
\label{fig:training}
\end{figure}

In overview, mixup barcodes capture how the inclusion of another point cloud affects the persistent homology of an existing point cloud. Specifically, we ask how 
the lifetime of each topological feature shortens in the presence of new data. This way we quantify the spatial relationships between the two point clouds. While this task fits the general framework of multi-parameter persistent homology, we take a more pragmatic and computable approach. We frame this intuition in terms of persistent homology as well its variant, image persistent homology, which was recently shown to be efficiently computable~\cite{bauer2023efficient}. We carefully combine these two pieces of information.



\myparagraph{A proof of concept application.} Our method is motivated by a problem arising in machine learning. We will refer to this problem as \emph{the disentanglement problem}. In~\cite{olah2014manifolds}, it was hypothesized that training can be hindered by geometric-topological interactions within the intermediate representations (point clouds) produced by such models. Intuitively, a model may fail to \new{disentangle} the representations corresponding to data with different labels -- see \cref{fig:training} for a concrete example. First steps towards testing this hypothesis were recently made using persistent homology~\cite{naitzat2020topology, wheeler2021activation}. We argue that our method offers additional, useful information, and present experiments supporting this claim. We elaborate on the experimental setup in Section~\ref{sec:app}. 

\myparagraph{Contributions.}
The contributions of this paper include:
\begin{enumerate}
    \item Mathematical foundations for a new topological descriptor called a \emph{mixup barcode} characterizing spatial relationships between point clouds.    
    \item Summary statistics quantifying the robustness of these relationships.        
    \item Experiments showing the computability and usefulness of mixup barcodes in the above-mentioned machine-learning problem.     
    \item A simple practical algorithm for computing mixup barcodes.
    \item This is one of the first applications of image persistence~\cite{cohen2009persistent} and the first application we are aware of for high-dimensional data.  
\end{enumerate}

\myparagraph{Connections and potential impacts.}
We expect this work to expand the standard paradigm of topological data analysis,
namely from summarizing and learning from the shape of a point cloud,
to summarizing and learning from the spatial relationships among point clouds. 
This goal is aligned with the emerging direction of Chromatic Topological Data Analysis outlined in~\cite{di2024chromatic}. Our work was developed independently, but both 
directions build on the same theoretical and algorithmic foundation~\cite{cohen2009persistent}.
The main novelty of our approach comes from providing a single, compound descriptor that merges two simpler barcodes. These barcodes are typically considered in isolation -- which loses important information that our method retains (at a negligible computational cost). We will contrast these approaches further in~\cref{sec:discussion}

Once the theoretical and computational aspects are fully developed, the proposed tool 
will likely open new applications of geometric-topological methods. 


\section{Related work} \label{sec:related}
We briefly review related work in computational geometry and topology, as well as their relevant applications in machine learning.


\myparagraph{Image persistence.}
Regarding computational methods for image persistence, an algorithm was proposed by Cohen-Steiner and collaborators~\cite{cohen2009persistent}. An implementation is available in the first version of Morozov's Dionysus library, but  not in the current version. Recently, Bauer and Schmahl proposed a version~\cite{bauer2023efficient} of the Ripser library for image persistence of Vietoris--Rips complexes, which informs our implementation. 
One of the first practical applications of image persistence was done recently in the 
context of image segmentation~\cite{stucki2023topologically}.

Cultrera di Montesano and collaborators~\cite{di2022persistent} have proposed an approach for analyzing the mingling of a small number of 3-dimensional point clouds using image (as well as kernel and cokernel) persistence. One crucial development is the chromatic alpha complex which allows for the efficient computation of these variants of persistence for data embedded in low dimensions. Their motivation is similar to ours, although both projects were started and developed independently. Our focus is on a single descriptor combining standard persistence with image persistence, while theirs is on different types of persistent homology and connections between them. Natarajan et al., study the Morse theory of chromatic Delaunay complexes~\cite{natarajan2024morse} and provide an implementation.

\myparagraph{Induced matching.}
 Our results are also related to work on induced matchings originating in~\cite{bauer2015induced}. In particular, similar 
 ideas were independently explored by Gonzalez-Diaz and coauthors~\cite{gonzalez2020additive, gonzalez2023topological, gonzalez2021partial}.

\myparagraph{Persistent homology in machine learning.} Techniques and tools from topological data analysis have recently found a number of fruitful applications to machine learning and in particular deep neural networks~\cite{zia2023topological}. Persistent homology has been used in developing topological network layers~\cite{bruelgabrielsson2020topology, hofer2017deep,  kim2020efficient}. Another approach incorporates  topological information about the data into the loss function of a neural network~\cite{chen2019topreg, wang2020topogan,  hu2019topology, clough2022toploss}. Topological methods were also used to analyze other aspects of artificial neural networks~\cite{rieck2018neural, Rathore2019, zheng2021topological, watanabe_topological_2022,bruelgabrielsson2019exposition, yan2021link}. In \cite{barannikov2021representation} a variant of persistent homology was developed to compare neural networks.

\myparagraph{Disentanglement in deep neural networks.}
In machine learning, disentanglement is a broad concept. We are interested in  how embeddings corresponding to different classes are separated during the training of a classifier. A motivation for this work goes back to an influential 2014 blog post~\cite{olah2014manifolds} by Olah on "Neural Networks, Manifolds, and Topology"  -- ten years later we have appropriate computational tools to address this challenge. 

Perhaps the fist explicit mention of (manifold) disentanglement in the context of deep learning is in the work of Brahma, Wu and Yiyuan~\cite{fsuftw}. The work of Zhou and collaborators provides a topological view on disentanglement in the context of machine learning as well as a thorough overview~\cite{zhou2020evaluating}. We refer the reader interested in the machine learning details there, and focus on geometric-topological aspects of the problem. Our experiments are inspired by the work by Naitzat and collaborators~\cite{naitzat2020topology} and the follow-up by Wheeler, Bouza and Bubenik~\cite{wheeler2021activation}. We alleviate some of the shortcomings of these approaches that stem from the limitations of persistent homology applied to a single point cloud. Reani and Bobrowski proposed a data-analysis method~\cite{reani2022cycle} based on matching geometrically close cycles.

\section{Mixup barcodes} \label{sec:setup}
In this technical section we propose the mathematical setup for mixup barcodes. We delay detailed interpretation to the next section. We will work with filtrations of simplicial complexes, chain complexes, persistence modules and their decompositions. 

\myparagraph{Overview.}
In the simple case, we study point clouds $A$ and $B$, via the inclusion $A \hookrightarrow A \cup B$. We ask how 
the persistence of each topological feature arising from $A$ is shortened by the inclusion of points in $B$. The constructed \emph{mixup barcode} 
splits each bar in the standard persistence barcode of $A$ into two sub-bars. The first sub-bar quantifies this new, shortened lifetime, 
and is computed using a variant of persistent homology called image persistence. The second sub-bar is called a mixup sub-bar,
and quantifies the shortening of the lifetime.

We follow up with technical definitions,
but a reader interested in an intuitive explanation
could jump to~\cref{sec:int} and especially~\cref{fig:ex3d}.

\myparagraph{Technical definitions.} We will work with two filtrations $L$ and $K$ -- which will be connected by an inclusion that we describe in a moment.
Specifically, let $K$ denote a filtered finite simplicial complex $K_1 \hookrightarrow K_2 \hookrightarrow \cdots \hookrightarrow K_n$ with inclusion maps 
$K_i \hookrightarrow K_j$ for $1 \leq i \leq j \leq n$, and similarly for $L$.

We assume that both filtrations are simplex-wise, i.e. 
for each $i$, $L_{i+1}$ differs from $L_i$ by at most one simplex and similarly for $K$. 
We further assume that each $\sigma$ added in $L_i$, is not already present in $K$, namely $\sigma \not\in K_{i-1}$, where $K_0 = \emptyset$.
The index of a simplex is the $i$ at which it appears in the filtration, which is uniquely defined under the above assumptions.
Further, we will sometimes assume that we have a real-valued order-preserving function $f$ assigning a \new{filtration value} to each simplex. 

The above assumptions may be met in the context of Vietoris--Rips filtrations, which will be our practical setup. In particular, while several simplices will typically appear at the same filtration value, it is easy to use an ordering of the simplices to obtain a finer simplex-wise filtration. 
We remark that everything in this section also generalizes to finite cell complexes.

\myparagraph{The studied inclusion map.} For each $k \geq 0$ we consider the inclusion map $\iota: L \hookrightarrow K$ between our filtrations $L$ and $K$. This gives rise to a variety of algebraic 
structures and induced maps between them.
Of primary interest are the  persistence modules $H_k(L)$ and $H_k(K)$ and an induced map between them $H_k(\iota): H_k(L) \to H_k(K)$. 
This information is summarized in the
following commutative diagram of vector spaces and linear maps. 

\begin{align}
\label{ladder}
\centering
\begin{CD}
H_k(L_1) @>>> H_k(L_2) @>>> \ldots @>>> H_k(L_i) @>>> \ldots @>>> H_k(L_n) \\
@VH_k(\iota_1)VV @VH_k(\iota_2)VV @. @VH_k(\iota_i)VV @. @VH_k(\iota_n)VV \\
H_k(K_1) @>>> H_k(K_2) @>>> \ldots @>>> H_k(K_i) @>>> \ldots @>>> H_k(K_n) 
\end{CD}
\end{align}

The image of the map of persistence modules in \eqref{ladder} is the following persistence module:
\begin{equation*}
    \im(H_k(\iota)) = \im(H_k(\iota_1)) \to \im(H_k(\iota_2)) \to \cdots \to \im( H_k(\iota_i)) \to \cdots \to \im(H_k(\iota_n))
\end{equation*}
 
We are interested in the barcodes of the persistence module $H_k(L)$ (i.e. the persistent homology of $L$) and the persistence module $\im(H_k(\iota))$ (i.e. the image persistent homology of the inclusion). Specifically, we will combine the information contained 
in these barcodes in a meaningful way, to form a single richer descriptor. 

To do so, we notice that under our assumptions for each index $i \in \{1,2\ldots,n\}$ there is at most one interval in the barcode with minimum value $i$. We call such $i$ the \new{birth index} of the corresponding interval in the barcode. This allows us to match these barcodes by the birth \emph{indices}; for brevity we refer to the following more general result.

\begin{theorem}[{Matching \cite[Theorem 4.2, Proposition 5.7]{bauer2015induced}}]
    There is a canonical 
    surjective matching from the barcode of $H_k(L)$ to the barcode of $\im(H_k(\iota))$ which matches intervals with the same birth index.
\end{theorem}
With this, we reorganize the information in this canonical 
matching of barcodes as follows.
\begin{definition}[Mixup barcode]
    The \new{mixup barcode} in degree $k$ of the inclusion $L \hookrightarrow K$ is the collection of triples of indices consisting of a  triple $(b,d',d)$ for each pair of matched intervals $[b,d)$ and $[b,d')$ in the barcodes of $H_k(L)$ and $\im(H_k(\iota))$; and a triple $(b,b,d)$ for each unmatched interval $[b,d)$ in the barcode of $H_k(L)$.
\end{definition}



The above definition is in terms of the filtration indices.
In practice, we often replace them with the filtration values of the corresponding simplex. 
Typically the distinction will be clear from context, but sometimes we will explicitly talk about \new{index mixup barcodes}. 

Viewing the mixup barcode as a refinement of the persistence barcode of $L$, we visualize it in a similar way. Specifically, we plot the image sub-bar in a light color, and its complement, the mixup sub-bar, in a darker color.  We consider a separate mixup barcode for each degree; see Figure~\ref{fig:ex3d} for an example.

\myparagraph{Mixup barcodes for point clouds.}
We consider two finite point clouds, $A,B \subset \mathbb{R}^d$,
or more generally, a finite metric space $X$ with $A \subset X$ and $B = X \setminus A$,
and construct Vietoris--Rips filtrations of 
$A$ and $A \cup B$, with a common sequence of radii, $r_1 \leq r_2 \leq  \cdots \leq r_n$. 
The simplicial filtrations $L$ and $K$ are given by $L_i = \VR(A; r_i)$ and $K_i = \VR(A \cup B; r_i)$. 
By the mixup barcode of $A \hookrightarrow A \cup B$, we mean the mixup barcode
induced by $L \hookrightarrow K$ as defined above. Unlike a persistence barcode, a mixup barcode is sensitive to the \emph{ordering} of the simplices,
since it may affect the matching.

\section{Related concepts and statistics}
With the main technical definition in place, we 
explain the rationale behind it, and mention some properties. We start with defining an analog of a \new{persistence pair}~\cite{edelsbrunner2010computational} that can be expressed either in terms of the indices or filtration values.
\begin{definition}[Mixup triple]
Each triple $(b,d',d)$ in the mixup barcode of $L \hookrightarrow K$
is called an \new{mixup triple}.
\end{definition}

We reiterate what information is stored in each mixup triple.
\begin{definition}[Persistence bar and image and mixup sub-bars]
For each mixup triple $(b,d',d)$, we have a \new{persistence bar} $[b,d)$, which splits into what we call:
\begin{enumerate}[i]
    \item an \new{image sub-bar} $[b,\, d')$, and
    \item a \new{mixup sub-bar} $[d',\,d)$,  
\end{enumerate}
noting that degenerate sub-bars of the form $[x,x)$ are allowed.
\end{definition}

\begin{obs}
    The mixup barcode of $L \hookrightarrow K$ contains the persistence bars of $L$ and its image sub-bars correspond to the image persistence barcode. However,
    the mixup sub-bars will generally \emph{not} coincide with any bars of kernel persistence
    barcode. 
     See Figure~\ref{fig:ex1} in Appendix~\ref{app:examples} for an explicit example.
\end{obs}


\begin{definition}[Premature death]
Given an index mixup triple $(b,d',d)$ we interpret the index $d'$ as the index of premature death of the unique homology class $\gamma$ born at $L_b$. Specifically, we mean the first index  at which $\gamma$ becomes trivial or merges with an older homology class in $H_k(L)$.
\end{definition}

\begin{obs}
We stress that a premature death arises from the image persistent homology vector spaces of our inclusion, namely $\frac{Z_k(L_i)}{B_k(K_j) \cap Z_k(L_i)}$. (Contrast this with the standard persistence: $\frac{Z_k(L_i)}{B_k({L}_j) \cap Z_k(L_i)}$.) This premature death is therefore caused by the extra boundaries coming from $K$ -- but not by merging with another homology class present in $H_k(K)$ but absent from $H_k(L)$.     
\end{obs}


\subsection{Statistics}
It is often useful to extract a single number from a persistence barcode. 
For example, the total persistence -- the sum of the lengths of all bars -- is used to quantify the topological complexity of a point cloud. In a similar vein, we define simple ways of extracting a number to quantify the strength of interactions between filtrations $L$ and $K$.

\begin{definition}[Mixup]
The mixup of a mixup triple $(b,d',d)$ is the value $d - d'$.
\end{definition}
In other words, it is the length of the sub-bar $(d', d)$ which quantifies the prematurity of the death of a single homology class.

\begin{definition}[Total mixup]
The \new{total mixup} 
of $L \hookrightarrow K$
is the sum of the mixups for the bars in the mixup barcode.     
\end{definition}
In our visualization, it is simply the sum of the lengths of the dark bars. 
We refer to this sum as total-mixup$(A,B)$.
Equivalently, it is the difference between the total persistence of $A$ and the total image persistence of $A \hookrightarrow A \cup B$. 



\myparagraph{Scale invariant statistics.}
The next two definitions allow us to make scale-invariant measurements. Firstly, this allows us to compare data embedded in different dimensions. Secondly, the geometric scale of topological features is irrelevant for our application.

\begin{definition}[Mixup percentage]
The \new{mixup percentage} of mixup triple $t = (b,d',d)$ is
\begin{align}
mixup_{\%}(t) = 
\frac{mixup(t)}{pers(t)} = \frac{d-d'}{d-b}.     
\end{align}    
\end{definition}

\begin{figure}
    \centering
    \includegraphics[trim=0cm 0.0cm 0.0cm 0.0cm, clip, width=0.7\linewidth]{figs/example_3d_open.png}
    \centering    
    \includegraphics[trim=0.75cm 0 0.89cm 0, clip, width=1.0\linewidth]{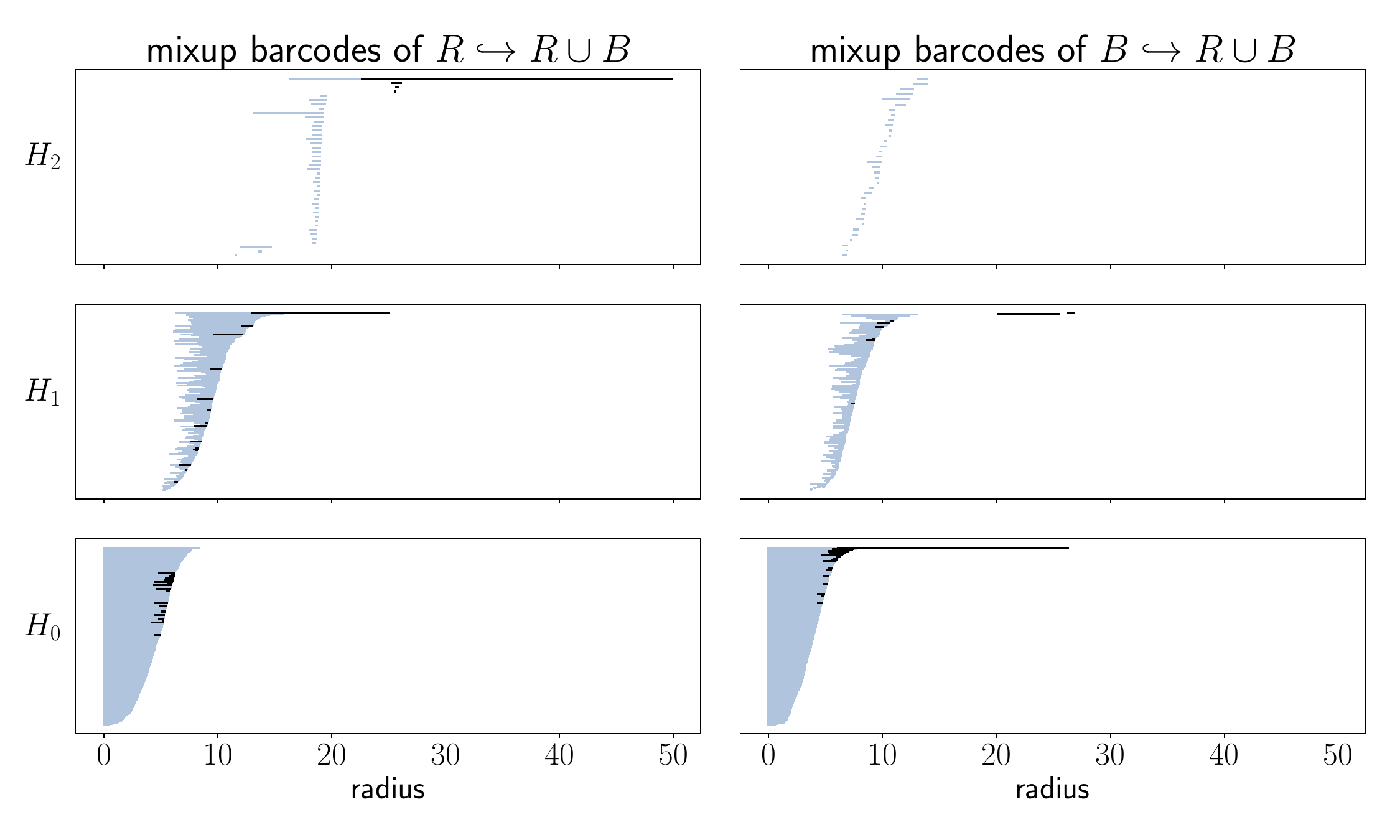}
    \caption{\figb{Top:} Two point clouds: $R$ red (dark) and $B$ blue (light). \figb{Bottom:} Mixup barcodes of $R$ and $B$ included in their union, in degrees $2,1,0$. 
    Recall that each persistence bar is split into an image sub-bar and a mixup sub-bar (marked in thick black).
    \figb{Bottom left:}  In $H_2$, we see a prominent persistence bar with a long mixup bar. It detects the red \new{void} partially filled by blue points, interpreted as the red object \new{surrounding} a part of the blue object. Similarly in $H_1$ the prominent mixup bar detects the red handle (tunnel) that \new{encircles} the blue points. No significant mixup occurs in $H_0$, although the many shorter mixup bars detect the \new{overlap} between the two shapes. \figb{Bottom right:} In $H_1$ the mixup barcode detects the handle \new{encircled} by the red object, with the late birth indicating that a large part is missing. The prominent bar for $H_0$ detects that the two blue connected components connect quicker via the red points. We interpret this as \new{separation} of the two blue parts by the red object. As expected, there is no prominent persistence in $H_2$.}    
    \label{fig:ex3d}
\end{figure}

\begin{definition}[Total mixup percentage]
The \new{total mixup percentage} (\new{mean mixup percentage}) of a mixup barcode is the sum (mean) of the mixup percentages for all its bars.
\end{definition}

Note that the latter variant is between 0 and 1, but is \emph{not} the ratio of the total mixup and the total persistence. The mean mixup percentage has the added benefit of allowing for the comparison between datasets of differing total persistence, which will be crucial in
our application.

\section{Geometric interpretation}
\label{sec:int}
In this section we provide intuitive interpretations of the information captured by mixup barcodes in various degrees. It turns out that they often capture natural spatial relationships. 

\myparagraph{Holes.}
We recall that for a shape embedded in three dimensional space, homology groups in \new{degree} $0$, $1$, $2$ are often interpreted as various types of holes: the gaps between distinct parts of the object (or dually the connected components); tunnels (or handles); voids enclosed by the object. Note that in general voids correspond to $H_{n-1}$ for objects in $\mathbb{R}^n$. Persistent homology captures the births and deaths of such features along the filtration.

\myparagraph{Spatial relationships for point clouds.}
Going back to the case of a pair of point clouds $A$ and $B$ in $\mathbb{R}^3$, the shortening of the lifetime of a homological feature conveys how points from $B$ fill the holes of different types. 

Specifically, in degree $0$, mixup quantifies to what extent the connected components in $A$ \emph{merge quicker} via points in $B$ -- which has natural interpretations as $A$ and $B$ \new{overlapping} or $A$ being \new{separated} by $B$. In degree $1$, it quantifies how much quicker points in $B$ fill up a loop or tunnel formed by $A$ -- for example when a portion of $A$ encircles a portion of $B$ (like in a Hopf link). In degree $2$, it quantifies how points in $B$ contribute to filling up a void inside $A$ -- which we interpret as a portion of $A$ \new{surrounding} a portion of $B$. Figure \ref{fig:ex3d} illustrates these spatial relationships and the corresponding mixup barcode. 

For the machine learning example illustrated in Figure~\ref{fig:training}, mixup in degree $1$ is most relevant as it captures the red points initially \emph{encircling} the blue points,
and then either \emph{tightening} around them, or \emph{separating} from them. For the second case, mixup in degree $0$ is also useful.

\begin{quote}
In summary, for point clouds in $\mathbb{R}^n$, the mixup barcode captures: \new{overlap} or \new{separation} in degree $0$ and \new{surrounding} in degree $n-1$. In remaining degrees it captures how holes in degree $k$ present in $L$ are filled by boundaries in $K$, which corresponds to more subtle and harder to name interactions.
\end{quote}



\section{Algorithm}
\label{sec:alg}
In this section we describe a simple
algorithm for the decomposition of a persistence barcode into a mixup barcode and prove its correctness under the assumptions in Section~\ref{sec:setup}. We start by recalling the algorithms for standard and image persistent homology, which will be part of our construction. 

\myparagraph{Standard reduction algorithm.}
Persistent homology can be computed using a version of Gaussian elimination on the boundary matrix representing the filtration~\cite{edelsbrunner2010computational}.
We call such an algorithm $\textsc{REDUCE}$.


The ordering of columns and rows corresponds to the order in which the elements (typically simplices) 
appear in the filtration. By adding the columns (only from left to right) modulo 2, the algorithm removes conflicts between columns, by which we mean columns having the same pivot. The index persistence pairs are formed by the indices of nonzero columns and their pivots 
in the final (reduced) matrix. This information does not depend 
on the choice of column additions~\cite{edelsbrunner2000topological}.



\myparagraph{Image persistent homology computations.}
Image persistent homology can be computed using a variant of the above algorithm, as described in~\cite[pp. 5-6]{cohen2009persistent}. Perhaps surprisingly, this algorithm reorders the \emph{rows} of the boundary matrix, so that the elements from $L$ come first (ordered by the filtration index), followed by the 
elements from $K \setminus L$ (also ordered by the filtration index). The image index persistence pairs are of the form $(\sigma, \tau')$, where the $\tau'$ is the index of a nonempty column whose pivot $\sigma$ corresponds to a simplex in $L$~\cite[Observations (ii),(iii)]{cohen2009persistent}. 
This yields a modified Elder rule: The youngest class from $K \setminus L$ is destroyed whenever available; otherwise, the youngest class from $L$ is destroyed, as usual~\cite{edelsbrunner2010computational}. A cohomological version is obtained by working on anti-transposed matrices~\cite{PHAT}.


\myparagraph{Mixup barcodes algorithm.} \cref{alg:mixup-decomposition} outlines the algorithm that decomposes the bars in the standard persistence diagram into the mixup barcode. The key new idea is to coordinate the two matrix reductions, but preparing two filtrations with consistent indexing. These reductions yield the standard and image persistent homology. The latter computation involves the reordering of the rows, as mentioned above. This coordinated reduction allows us to  merge the two pieces of information. In contrast, computing the standard and image persistence on two unrelated filtrations loses the correspondence between these two pieces of information. Additionally, we reiterate that mixup bars generally do not coincide with kernel persistence, necessitating the use of our algorithm.

\myparagraph{Computational complexity and implementation.}
Due to the use of Gaussian elimination, the algorithm runs in $O(n^3)$ worst-case time, where $n$ is the number of simplices. As with standard persistence~\cite{PHAT}, the running time is significantly better in practice. 

\begin{algorithm}[H]
\caption{Algorithm for Mixup Barcode in degree $k$}
\label{alg:mixup-decomposition}
\begin{algorithmic}[1]
\Require Simplex-wise filtrations $L$ and $K$ as in Section~\ref{sec:setup}; $f$ assigning the filtration value to each simplex; degree $k$

\State $BK = $ matrix with columns containing the boundaries of the $k$- and $(k+1)$-simplices of $K$ with rows and columns ordered by the filtration index
\State reorder the rows of $BK$ such that the simplices from $L$ have smaller indices than the simplices of $K \setminus L$ while retaining the relative order of simplices in $L$ and also in $K \setminus L$
\State form matrix $BL$ by zeroing out the entries in $BK$ involving the simplices of $K \setminus L$
\State \textsc{REDUCE} matrices $BL$ and $BK$.
\For{each index $\sigma$ of a $k$-simplex in $L$ such that $BL[\sigma] = 0$}
    \State $birth = f(\sigma)$; $death = death' = \infty$    
    \If {$\exists\ \tau$ such that $\text{pivot} (BL[\tau]) = \sigma$}
    $death = f(\tau)$  \EndIf
    \If {$\exists\ \tau'$ such that $\text{pivot}(BK[\tau']) = \sigma$}
    $death' = f(\tau')$  \EndIf    
    \State record mixup triple $(birth, death', death)$    
\EndFor
\State return all the recorded mixup triples
\end{algorithmic}
\end{algorithm}

\begin{theorem}[Mixup Decomposition Theorem]
\label{thm:decomp}
The algorithm in~\cref{alg:mixup-decomposition}  returns a correct 
mixup barcode of $L \hookrightarrow K$.
\end{theorem}
We prove this theorem in Appendix~\ref{app:proof}.

\section{Experiments}  \label{sec:app}
We apply our method to analyzing high-dimensional point clouds. They come from a geometric problem arising in machine learning.
The goal is to gain insights about the training of an artificial neural network~\cite{amariMLP67}, the \new{multilayer perceptron} (MLP). It is the simplest \emph{deep learning} method, often used for classification. We remark that models based solely on MLPs can achieve state of the art performance on challenging tasks~\cite{tolstikhin2021mlp}. 


\myparagraph{Model parameters.}
We use a 5-layer MLP, with embedding dimensions $512,256,128,10$ and $L$, where $L$ depends on the number of considered labels (in our case 3 or 10). The model has $565,248 + 10L$ trainable parameters (weights).


\subsection{Datasets}
We use two well known datasets with different levels of complexity. Note that while performing classification on these datasets is 
considered solved -- understanding how machine learning models achieve their good performance is an open problem.
\label{app:datasets}

\myparagraph{MNIST.}
This is a standard dataset, containing 60,000 grayscale images of size $28 \times 28$ depicting handwritten decimal digits. It is a heavily preprocessed version of real-world images, so that the Euclidean distance between two images closely corresponds to their perceptual dissimilarity. For this reason, representing the images as a (labeled) point cloud in $\mathbb{R}^{784}$ yields a faithful representation of the data. It comes split into 50,000 training images, and 10,000 test images. We normalize the values to be between 0 and 1. The model achieves 99\% test accuracy when trained and tested on the images with labels 0,1 and 2, which we then use for our experiments. These are images of handwritten zeros, ones and twos. 
It is known to be an easy problem so we expect to see low mixup at the end of training. We are however curious about the higher-dimensional embeddings coming from the intermediate layers.

\myparagraph{CIFAR10.}
Another standard dataset, containing 50,000 color photos of size $32 \times 32$ depicting various objects, animals etc. We treat this data as a (labeled) point cloud in $\mathbb{R}^{3072}$, but the Euclidean distance fails to capture the perceptual similarity. 
We expect to see more significant entanglement between different classes. This dataset is quite challenging for a basic MLP which achieves only 76\% test accuracy when trained and tested the images with labels 0,1 and 2. These are images of airplanes, cars and birds. We expect it to fail to disentangle the data, especially differentiating between airplanes and birds may be hard due to similar backgrounds.

\subsection{Experiments} \label{sec:experiments}
We use the new method to analyze geometric data arising from a machine learning model. The goal is to quantify geometric-topological difficulties during training -- as theoretically outlined by Olah~\cite{olah2014manifolds}. The results highlight the usefulness of the new methodology -- especially in comparison to existing applications of persistent homology~\cite{naitzat2020topology, wheeler2021activation} that focused on quantifying the shape of a single point cloud.  

In geometric terms, such a model learns a sequence of linear transformations each followed by a projection onto the positive orthant of $\mathbb{R}^m$, where $m$ may be different at each step. This is done in an iterative training process, whose goal is to finally embed each input point 
with label number $q$ near the $q$-th standard basis vector.

\myparagraph{Goals.}  We consider the intermediate representation of the data produced by the $i$ initial layers of the model, namely $M_k = \ell_k \circ \ell_{k-1} \circ ... \circ \ell_1$.  Given a single input vector $v$, we consider $(M_k(v))_k$, namely the sequence of embeddings of $v$ produced by each $M_k$. Extending this to the entire input point cloud, we get a sequence of point clouds -- which additionally change as training progresses. Figure~\ref{fig:training} illustrates this for a simplified three-dimensional case.

\begin{quote}
We analyze how the representation of input data disentangles passing through the layers -- and track this process during training. We aim to verify the conjecture that complicated spatial interactions can hinder training. 
\end{quote}

We return to the training depicted in Figure~\ref{fig:training}. In one of the two cases,
the model failed to fully disentangle the two point clouds, even though both training and test accuracy reach almost $100\%$. Such a model is problematic. In this particular case, consider new examples corresponding to the blue region around which the red loop was constricted. These examples are likely to be misclassified, and can for example be exploited in an adversarial attack. More generally, this example shows that 
geometric-topological information about the embeddings can reveal information 
about the quality of the training, even when the accuracy (and by analogy the value of the loss function) suggest no such problems.

This situation is of course idealized: the intermediate embedding dimensions can be arbitrarily high, which simplifies the disentanglement problem. Still, complicated entanglement is hypothesized~\cite{olah2014manifolds, zhou2020evaluating} to correlate with training difficulties such as over-fitting and lack of generalization. Indeed, even if 
disentanglement is theoretically possible, it does not mean that the simplistic training
process (oblivious to the global geometric-topological relationships) can realize it. Case in point: in our example, the problem was solvable in this dimension, but one of the training runs failed.

\myparagraph{Datasets.} To train the model, we will be using two datasets MNIST (easier, $99\%$ accuracy) and CIFAR10 (harder, $76\%$ accuracy). See~\cref{app:datasets} for more details. After training, we consider the embeddings produced from the test data. We treat the datasets fa finite labeled point cloud in $X \subset \mathbb{R}^d$. By $X_q$ we mean \new{class $q$} of dataset $X$, namely the subset of points with label number $q$.

\myparagraph{Computations and subsampling.}

In our experiments we used a modified version of Ripser, specifically the standard version~\cite{ripser} and image persistence version~\cite{bauer2023efficient}. 
As reported in~\cite{bauer2023efficient}, computing image persistence with Ripser is often slower than computing standard persistence for data of similar size and complexity. 
We therefore subsample $A$ with 500 points and $B$ with 100 points. We subsample consistently for all training epochs and layers. In degree $0$, the computations are faster, and we use all available points. Our preliminary results show that 
efficiency of image persistence can be significantly improved, which this is beyond the scope of this work.

\begin{lemma}[Subsampling Property] \label{obs:subs}
Given point clouds $A$, $B$ and $B' \subset B$, total-mixup($A, B'$) $\leq$ total-mixup($A, B$) (similarly for the mixup-percentage).
\end{lemma}
\begin{proof}
While the birth and death of a homology class arising from $A$ is unaffected by $B$, its premature death can only be \emph{delayed} by using a subsample of $B$. Indeed, there are fewer boundaries coming from $\VR(A\cup B'; r)$ than from $\VR(A\cup B; r)$ that can shorten its life.
\end{proof}
In other words we only risk that the captured signal will be weaker.

To subsample, we can therefore safely use the k-medoids algorithm~\cite{kaufman1990partitioning}, the long-lost sibling of the popular k-means clustering~\cite{macqueen1967some}. Unlike k-means, it ensures that the cluster centers belong to the input point cloud, so we can indeed use it for subsampling. It was used in a similar context by Li and collaborators~\cite{li2022selecting}. It is a good choice also in our case, because it is not crucial to carefully represent the shape of the points in $B$, since it has no bearing on the result. However, we aim to preserve the outliers. Indeed, even a single point in $B$ surrounded by points in $A$ may have a large impact on the topological mixup between $A$ and $B$.
In the context of our application, this may indicate 
a training example that is likely to be misclassified.


\begin{figure}[t]
    \centering
    \includegraphics[width=0.49\textwidth]{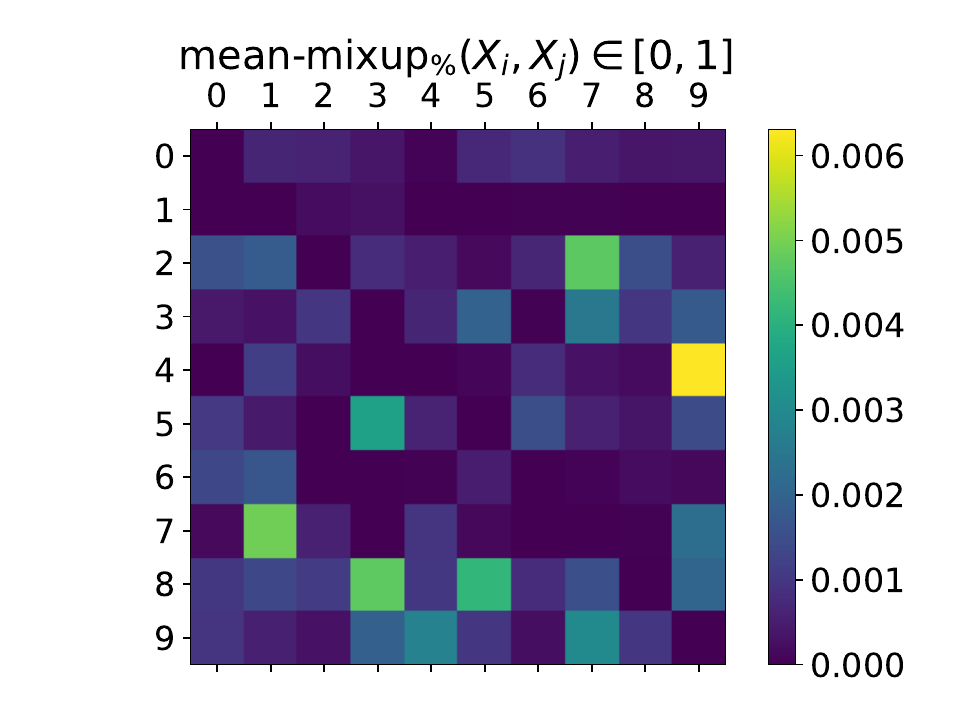}
    \includegraphics[width=0.49\textwidth]{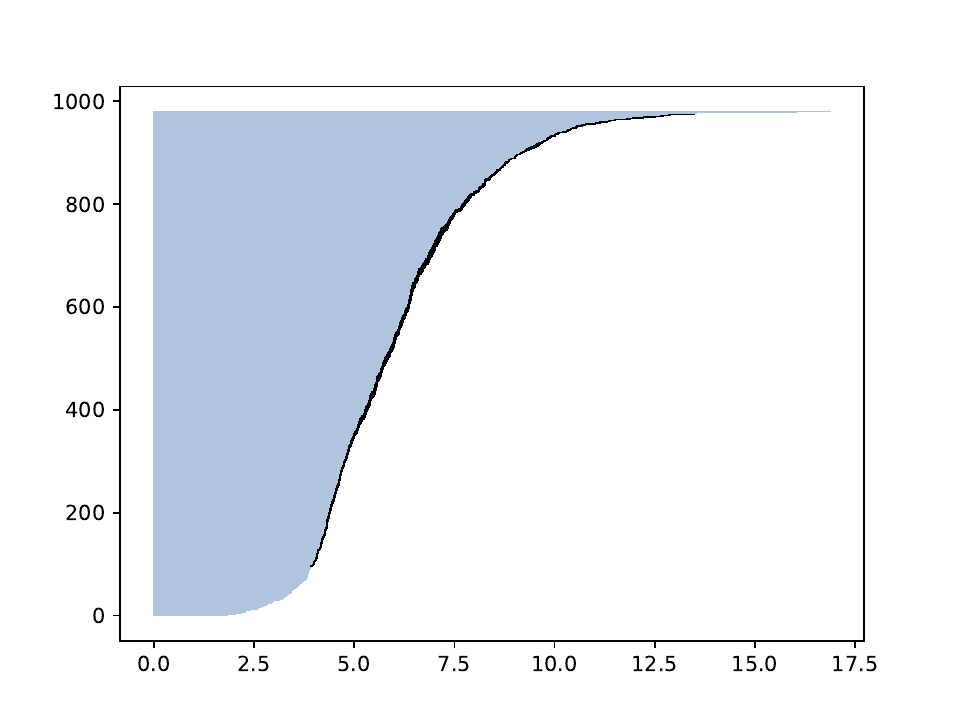}
    \caption{\figb{Left:} The mean mixup percentage (in degree 0) between all pairs of MNIST classes. Most values are low, as expected due to the simplicity of this data. \figb{Right}: The mixup barcode between the images of fours and nines, which have the greatest mixup. Many bars with small -- but positive -- mixup suggest that the two point clouds overlap along a broad but shallow interface.}
    \label{fig:mnist_10x10}
\end{figure}

\begin{figure}[!ht]
    \centering
    \includegraphics[width=0.49\textwidth]{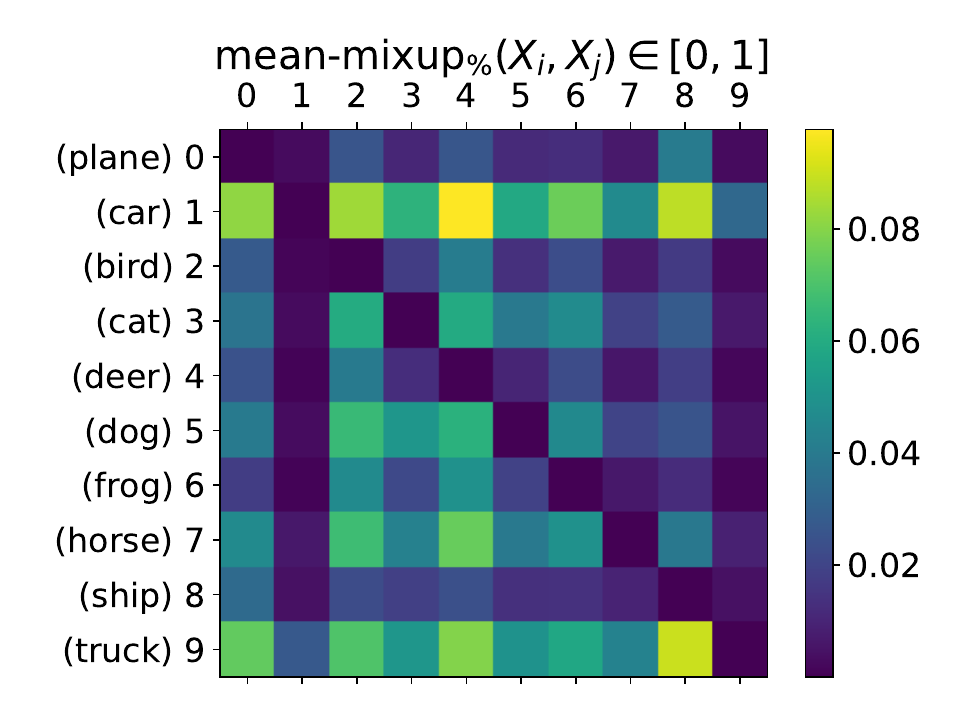}
    \includegraphics[width=0.49\textwidth]{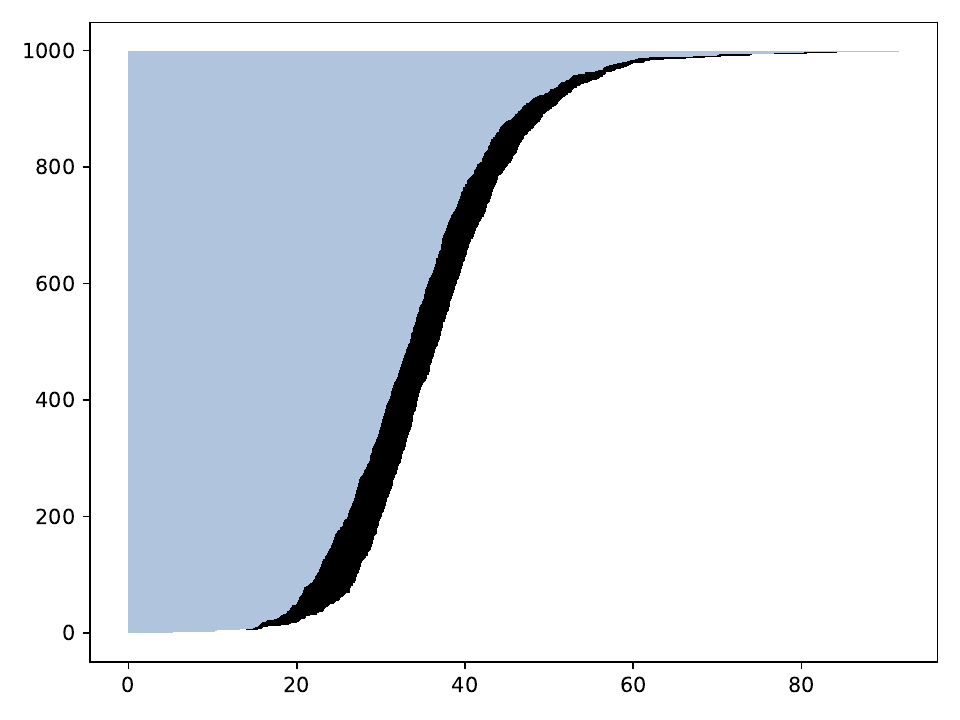}
    \caption{\figb{Left:} The mean mixup percentage (in degree 0) between all pairs of CIFAR10 classes. Note that the values are generally an order of magnitude higher than for the other dataset, with the highest value of 0.09 between classes 1 (car) and 4 (deer).  \figb{Right:} The mixup barcode between classes 1 and 4, suggesting a significantly more robust overlap than we saw in the other dataset.}
    \label{fig:cifar_10x10}
\end{figure}

\subsection{Warmup: mixup barcodes on raw inputs}
While our primary focus is to analyze the model (and not the input data), we focus on the raw MNIST data for a moment. We do this because previous studies considered the disentanglement of different classes of this dataset~\cite{wheeler2021activation} through the lens of persistent homology. 

However, strong performance of (non-kernel) linear classifiers on MNIST~\cite{yu2010improved, decoste2003fast} suggests that pairs of point clouds belonging to different classes are easy to separate using an affine hyperplane. This suggests that they are unlikely to be strongly entangled. 

As a sanity check, we verified that despite complicated topology, the mixup (which we use as a proxy for entanglement) is typically low. Specifically, we compute the mean percentage mixup in degree $0$ for all pairs of classes of examples, namely for $A = X_i$ and $B = X_j$ for all $0 \le i,j < 10$. The results for MNIST and CIFAR10 are in Figures~\ref{fig:mnist_10x10} and~\ref{fig:cifar_10x10}. 

As expected, CIFAR10 exhibits greater mixup, often by an order of magnitude. 
It is easy to see that persistence barcodes do not distinguish between the two cases. (The only difference comes from different dimensions of the two datasets.)

\begin{quote}
We emphasize that in general persistent homology of the data may be complicated, even if subsets of point clouds do not exhibit complex spatial relationships. Using persistence directly may overestimate the amount of entanglement. The extra information present in the mixup barcodes is therefore crucial for robust characterization of such phenomena.
\end{quote}

\begin{figure}[!ht]
    \centering    
    \begin{minipage}[b]{0.49\textwidth}
        \includegraphics[width = \textwidth]{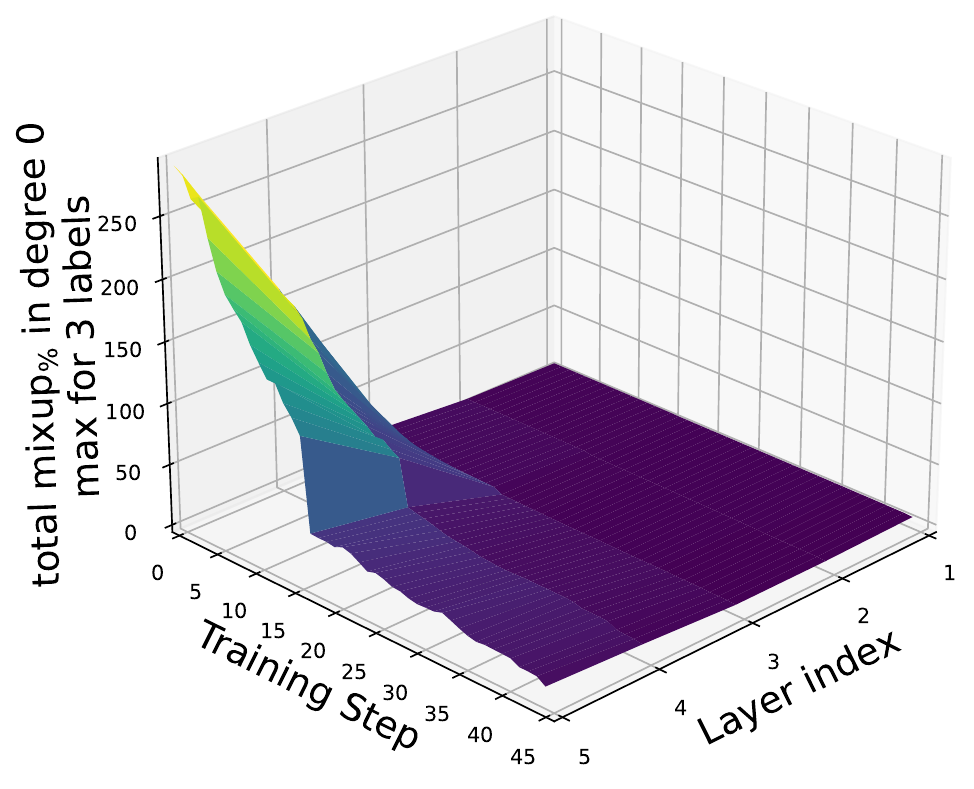} 
    \end{minipage}
    \hfill
    \begin{minipage}[b]{0.49\textwidth}
        \includegraphics[width = \textwidth]{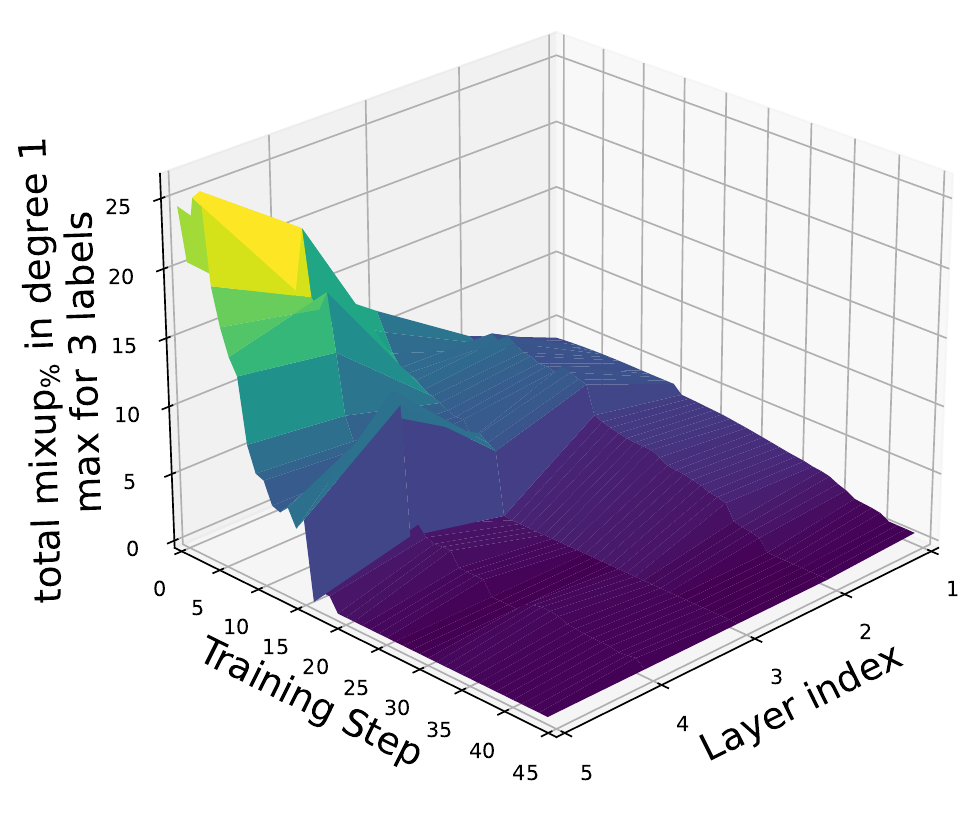}
    \end{minipage}
    
    \caption{Mixup profiles characterizing the training on the MNIST dataset, in degree 0 (\figb{left}) and degree 1 (\figb{right}). We recall that the height corresponds to the mixup -- the complexity of geometric-topological relationships -- between intermediate embeddings of data with different labels at a given training step and layer of the model. The point closest to the viewer corresponds to the last layer at the end of training. Here, the mixup initially \emph{grows} with consecutive layers (which makes sense since the input data is simple, but the matrix in each layer is initially random). On the other hand, mixup drops throughout with successive training iterations. This reflects that the network successfully \emph{disentangles} all the embeddings.}
    \label{fig:mnist-mixup-over-time}

    \centering    
    \begin{minipage}[b]{0.49\textwidth}
        \includegraphics[width = \textwidth]{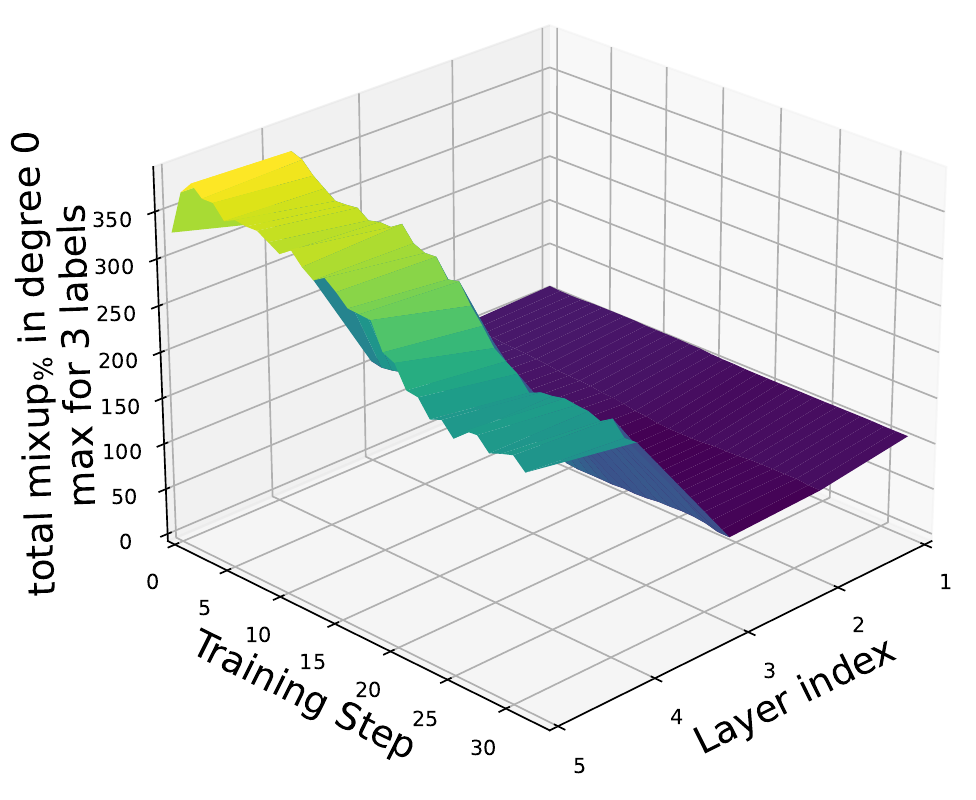} 
    \end{minipage}
    \hfill
    \begin{minipage}[b]{0.49\textwidth}
        \includegraphics[width = \textwidth]{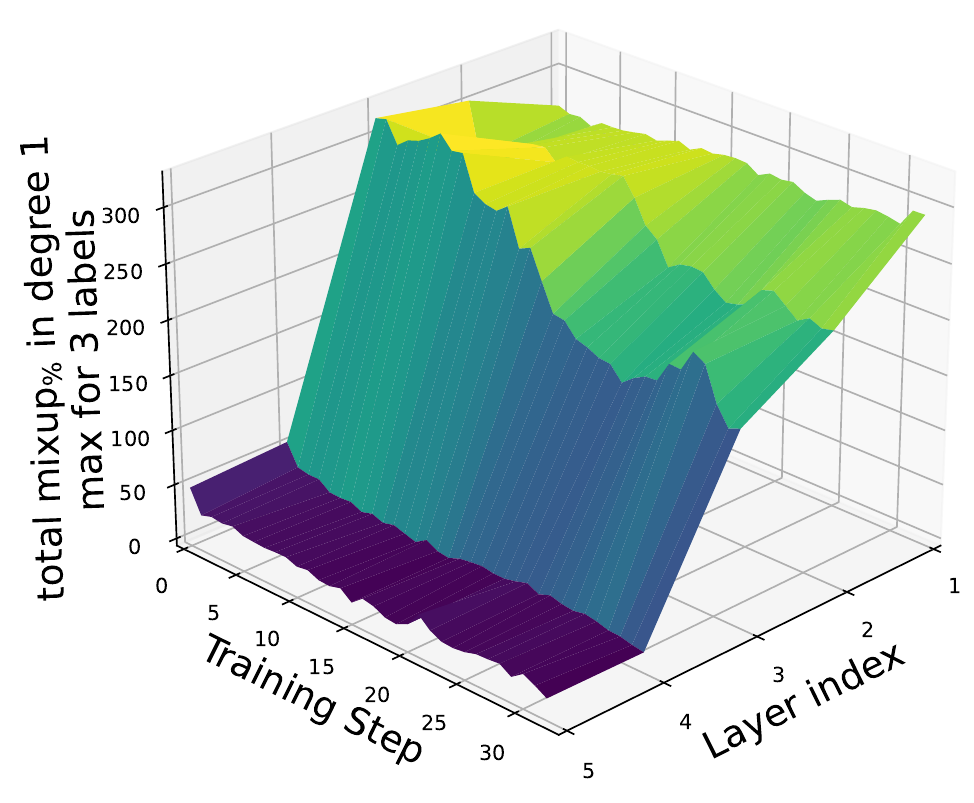}
    \end{minipage}
    
    \caption{Mixup profiles characterizing the training on the more challenging CIFAR dataset. While the total mixup percentage generally drops during training, it remains generally much higher for CIFAR compared to MNIST. This is consistent with the worse performance of this model on the CIFAR dataset.     
    In particular, for the last layer, mixup in degree $0$ initially drops but stabilizes at a relatively high value (which is consistent with the partially successful training); in degree $1$ it remains constant (and much higher 
    than for MNIST) -- this hints at more complicated spatial relationships between inputs with different classes that were hard to disentangle. Overall, the graphs suggest that the information captured in the mixup barcodes does indeed correlate with training difficulties and the model's failure to disentangle the embeddings.
    }
    \label{fig:cifar10-mixup-over-time}
\end{figure}

\subsection{Characterizing the disentanglement across layers}
\label{subsec:layers}
Our goal is to track disentanglement across layers during training. Specifically, we will analyze a labeled point cloud $X_{kqt} \subset \mathbb{R}^n$ for each layer $k$ ($1 \leq k \leq 5$), training epoch $t$, and label $q$ ($1 \leq q \leq 3$). The information about disentanglement will be 
characterized using the following concept.

\begin{definition}[Mixup profile]
At each layer $k$ and training epoch $t$, the mixup profile is computed as: $P_{kt} = \max_j \{ \text{total-mixup}_{\%}(X_{kjt}, \bigcup_{r \neq j} X_{krt}) \}$, in  homological degrees 0 and 1.
\end{definition}

The \new{mixup profile}, $P_{kt}$, measures complexity and strength of interactions between data representations with different labels, serving as a proxy for entanglement. It is visualized as a surface plot, with $k$ (layer) and $t$ (time) on the horizontal axes and $P_{kt}$ (strength of spatial relationship) as the height.

\myparagraph{Outcomes.}
See Figures~\ref{fig:mnist-mixup-over-time} and~\ref{fig:cifar10-mixup-over-time} for a direct visualization of the mixup profiles for the MNIST and CIFAR datasets. 

\begin{quote}
The main observation is that the mixup profiles are significantly different for the intermediate representations of the two datasets, suggesting that the interactions  captured by the mixup barcodes indeed correlate with training difficulties.
\end{quote}


\myparagraph{Outcomes summary.} The mixup profiles successfully distinguish between the training on the CIFAR and MNIST datasets.
In particular, the mixup quickly drops for the easier dataset (MNIST), and not for the harder dataset (CIFAR). Overall, the statistics derived from the mixup barcodes do correlate with difficulties 
during training. They therefore support the conjecture stated in~\cite{olah2014manifolds}. 
We stress that while MNIST and CIFAR10 are simple datasets, we analyze the training process on 
these datasets -- which is poorly understood, even for these datasets.
We discuss this further momentarily.

\section{Discussion} \label{sec:discussion}
We presented a novel geometric-topological descriptor called a mixup barcode, and its summary statistics. Unlike persistence barcodes, mixup barcodes can be used to characterize not only the shape of a dataset but also its spatial relationships with other datasets. While developed independently, it complements recent development of \new{Chromatic TDA}, an emerging branch of computational geometry and topology~\cite{di2024chromatic}. One advantage of our approach is that a single barcode is produced, which is often simpler to explain and analyze, especially compared to the six barcodes in~\cite{di2024chromatic}. Additionally, the mixup barcode retains the crucial matching (between the two relevant barcodes), which is otherwise lost. Note that this information cannot in general be reconstructed using individual barcodes, expressed in terms of filtration values. Indeed, the proposed algorithm, working at the level of \emph{index} persistence of two \emph{consistently indexed} filtrations, is necessary.  

We also reiterate that 
kernel persistence bars capture different information than our mixup sub-bars. 
For our application, it was crucial to use the total mixup percentage, which relies on the matching between persistence and image bars.

Our experiments mark the first application of Chromatic TDA ideas to analyzing high-dimensional data -- as well as the first practical application of image persistence to such data.
We believe image persistence (and related variants) are powerful tools, but currently underdeveloped and underutilized. We expect that other pairs (or perhaps larger subsets) of barcodes can be combined in a similar way.


On the applied side, this new tool sheds light on the complicated geometry and topology of the training process of machine learning models. The experiments suggest that the theoretical considerations proposed in~\cite{olah2014manifolds} (in terms of smooth manifolds, point-set topology and knot theory) can now be verified computationally (on finite point clouds). The main outcome is that mixup barcodes capture spatial interactions that were conjectured to hinder training~\cite{olah2014manifolds}. 

We expect that this new tool could help detect problems -- such as overfitting -- already during training, and regularize (or otherwise better control) the training to avoid such problems altogether.

However, further progress, including more detailed experiments, is
limited by underdeveloped computational tools and theory.
In particular a vectorization, for example generalizing persistence landscapes~\cite{bubenik2015statistical} or images~\cite{adams2017persistence}, would allow us to feed 
the mixup barcodes into a statistical or machine learning pipeline. Relatively slow computation~\cite{bauer2023efficient} of image persistence is another bottleneck. For now we are limited to simple descriptive statistics on small data. 

Still, we do believe that our preliminary experiments hint at the usefulness of mixup barcodes, in particular in the context of high-dimensional data.

\bibliographystyle{plainurl}
\bibliography{main}

\begin{thebibliography}{10}

\bibitem{adams2017persistence}
Henry Adams, Tegan Emerson, Michael Kirby, Rachel Neville, Chris Peterson, Patrick Shipman, Sofya Chepushtanova, Eric Hanson, Francis Motta, and Lori Ziegelmeier.
\newblock Persistence images: A stable vector representation of persistent homology.
\newblock {\em Journal of Machine Learning Research}, 18(8):1--35, 2017.

\bibitem{amariMLP67}
Shunichi Amari.
\newblock A theory of adaptive pattern classifiers.
\newblock {\em IEEE Transactions on Electronic Computers}, EC-16(3):299--307, 1967.
\newblock \href {https://doi.org/10.1109/PGEC.1967.264666} {\path{doi:10.1109/PGEC.1967.264666}}.

\bibitem{barannikov2021representation}
Serguei Barannikov, Ilya Trofimov, Nikita Balabin, and Evgeny Burnaev.
\newblock Representation topology divergence: A method for comparing neural network representations.
\newblock {\em arXiv preprint arXiv:2201.00058}, 2021.

\bibitem{ripser}
Ulrich Bauer.
\newblock Ripser: efficient computation of vietoris--rips persistence barcodes.
\newblock {\em Journal of Applied and Computational Topology}, 5(3):391--423, 2021.

\bibitem{PHAT}
Ulrich Bauer, Michael Kerber, Jan Reininghaus, and Hubert Wagner.
\newblock Phat: Persistent homology algorithms toolbox.
\newblock {\em Journal of Symbolic Computation}, 78:76 -- 90, 2017.
\newblock Algorithms and Software for Computational Topology.
\newblock URL: \url{http://www.sciencedirect.com/science/article/pii/S0747717116300098}, \href {https://doi.org/10.1016/j.jsc.2016.03.008} {\path{doi:10.1016/j.jsc.2016.03.008}}.

\bibitem{bauer2015induced}
Ulrich Bauer and Michael Lesnick.
\newblock Induced matchings and the algebraic stability of persistence barcodes.
\newblock {\em Journal of Computational Geometry}, 6(2):162--191, 2015.

\bibitem{bauer2023efficient}
Ulrich Bauer and Maximilian Schmahl.
\newblock Efficient computation of image persistence.
\newblock In {\em 39th International Symposium on Computational Geometry (SoCG 2023)}. Schloss Dagstuhl-Leibniz-Zentrum f{\"u}r Informatik, 2023.

\bibitem{fsuftw}
Pratik~Prabhanjan Brahma, Dapeng Wu, and Yiyuan She.
\newblock Why deep learning works: A manifold disentanglement perspective.
\newblock {\em IEEE Transactions on Neural Networks and Learning Systems}, 27(10):1997--2008, 2016.
\newblock \href {https://doi.org/10.1109/TNNLS.2015.2496947} {\path{doi:10.1109/TNNLS.2015.2496947}}.

\bibitem{bruelgabrielsson2019exposition}
Rickard Bruel~Gabrielsson and Gunnar Carlsson.
\newblock Exposition and interpretation of the topology of neural networks.
\newblock {\em Proceedings - 18th IEEE International Conference on Machine Learning and Applications, ICMLA 2019}, pages 1069--1076, 2019.
\newblock arXiv: 1810.03234 ISBN: 9781728145495.
\newblock \href {https://doi.org/10.1109/ICMLA.2019.00180} {\path{doi:10.1109/ICMLA.2019.00180}}.

\bibitem{bubenik2015statistical}
Peter Bubenik et~al.
\newblock Statistical topological data analysis using persistence landscapes.
\newblock {\em J. Mach. Learn. Res.}, 16(1):77--102, 2015.

\bibitem{chen2019topreg}
Chao Chen, Xiuyan Ni, Qinxun Bai, and Yusu Wang.
\newblock A topological regularizer for classifiers via persistent homology.
\newblock In Kamalika Chaudhuri and Masashi Sugiyama, editors, {\em Proceedings of the Twenty-Second International Conference on Artificial Intelligence and Statistics}, volume~89 of {\em Proceedings of Machine Learning Research}, pages 2573--2582. PMLR, 16--18 Apr 2019.
\newblock URL: \url{https://proceedings.mlr.press/v89/chen19g.html}.

\bibitem{clough2022toploss}
James~R. Clough, Nicholas Byrne, Ilkay Oksuz, Veronika~A. Zimmer, Julia~A. Schnabel, and Andrew~P. King.
\newblock A {Topological} {Loss} {Function} for {Deep}-{Learning} {Based} {Image} {Segmentation} {Using} {Persistent} {Homology}.
\newblock {\em IEEE Transactions on Pattern Analysis and Machine Intelligence}, 44(12):8766--8778, December 2022.
\newblock URL: \url{https://ieeexplore.ieee.org/document/9186664/}, \href {https://doi.org/10.1109/TPAMI.2020.3013679} {\path{doi:10.1109/TPAMI.2020.3013679}}.

\bibitem{cohen2009persistent}
David Cohen-Steiner, Herbert Edelsbrunner, John Harer, and Dmitriy Morozov.
\newblock Persistent homology for kernels, images, and cokernels.
\newblock In {\em Proceedings of the twentieth annual ACM-SIAM symposium on Discrete algorithms}, pages 1011--1020. SIAM, 2009.

\bibitem{vcufar2020ripserer}
Matija {\v{C}}ufar.
\newblock Ripserer. jl: flexible and efficient persistent homology computation in julia.
\newblock {\em Journal of Open Source Software}, 5(54):2614, 2020.

\bibitem{decoste2003fast}
Dennis DeCoste and Dominic Mazzoni.
\newblock Fast query-optimized kernel machine classification via incremental approximate nearest support vectors.
\newblock In {\em Proceedings of the 20th International Conference on Machine Learning (ICML)}, pages 115--122, Washington, DC, USA, 2003. AAAI Press.
\newblock Includes experiments on all 45 pairwise MNIST digit classification tasks; Table~1 reports linear SVM error rates on raw 784-D pixels for several pairs (e.g., 0 vs 1, 8 vs 3, 2 vs 5).

\bibitem{di2024chromatic}
Sebastiano~Cultrera di~Montesano, Ondrej Draganov, Herbert Edelsbrunner, and Morteza Saghafian.
\newblock Chromatic topological data analysis.
\newblock {\em arXiv preprint arXiv:2406.04102}, 2024.

\bibitem{di2022persistent}
Sebastiano~Cultrera di~Montesano, Ond{\v{r}}ej Draganov, Herbert Edelsbrunner, and Morteza Saghafian.
\newblock Chromatic alpha complexes.
\newblock {\em Foundations of Data Science}, 8:30--62, 2026.

\bibitem{edelsbrunner2010computational}
Herbert Edelsbrunner and John Harer.
\newblock {\em Computational topology: an introduction}.
\newblock American Mathematical Soc., 2010.

\bibitem{edelsbrunner2000topological}
Herbert Edelsbrunner, David Letscher, and Afra Zomorodian.
\newblock Topological persistence and simplification.
\newblock In {\em Proceedings 41st annual symposium on foundations of computer science}, pages 454--463. IEEE, 2000.

\bibitem{edelsbrunner2024maximum}
Herbert Edelsbrunner and J{\'a}nos Pach.
\newblock Maximum betti numbers of {\v{c}}ech complexes.
\newblock In {\em 40th International Symposium on Computational Geometry (SoCG 2024)}. Schloss Dagstuhl--Leibniz-Zentrum f{\"u}r Informatik, 2024.

\bibitem{bruelgabrielsson2020topology}
Rickard~Br\"uel Gabrielsson, Bradley~J. Nelson, Anjan Dwaraknath, and Primoz Skraba.
\newblock {A Topology Layer for Machine Learning}.
\newblock In Silvia Chiappa and Roberto Calandra, editors, {\em Proceedings of the Twenty Third International Conference on Artificial Intelligence and Statistics}, volume 108 of {\em Proceedings of Machine Learning Research}, pages 1553--1563. PMLR, 26--28 Aug 2020.
\newblock URL: \url{https://proceedings.mlr.press/v108/gabrielsson20a.html}.

\bibitem{gonzalez2021partial}
R.~Gonzalez-Diaz, M.~Soriano-Trigueros, and A.~Torras-Casas.
\newblock Partial matchings induced by morphisms between persistence modules.
\newblock {\em Computational Geometry}, 112:101985, 2023.
\newblock URL: \url{https://www.sciencedirect.com/science/article/pii/S0925772123000056}, \href {https://doi.org/10.1016/j.comgeo.2023.101985} {\path{doi:10.1016/j.comgeo.2023.101985}}.

\bibitem{gonzalez2020additive}
Roc{\'\i}o Gonz{\'a}lez-D{\'\i}az, Manuel Soriano-Trigueros, and {\'A}lvaro Torras-Casas.
\newblock Additive partial matchings induced by persistence morphisms.
\newblock {\em arXiv preprint arXiv:2006.11100}, 2020.

\bibitem{hofer2017deep}
Christoph Hofer, Roland Kwitt, Marc Niethammer, and Andreas Uhl.
\newblock Deep learning with topological signatures.
\newblock In {\em Proceedings of the 31st International Conference on Neural Information Processing Systems}, NIPS'17, page 1633–1643, Red Hook, NY, USA, 2017. Curran Associates Inc.

\bibitem{hu2019topology}
Xiaoling Hu, Fuxin Li, Dimitris Samaras, and Chao Chen.
\newblock Topology-preserving deep image segmentation.
\newblock {\em Advances in neural information processing systems}, 32, 2019.

\bibitem{kaufman1990partitioning}
Leonard Kaufman and Peter~J. Rousseeuw.
\newblock Partitioning around medoids (program pam).
\newblock In {\em Wiley Series in Probability and Statistics}, pages 68--125. John Wiley \& Sons, Inc., Hoboken, NJ, USA, 1990.
\newblock Retrieved 2021-06-13.
\newblock \href {https://doi.org/10.1002/9780470316801.ch2} {\path{doi:10.1002/9780470316801.ch2}}.

\bibitem{kim2020efficient}
Kwangho Kim, Jisu Kim, Manzil Zaheer, Joon~Sik Kim, Frederic Chazal, and Larry Wasserman.
\newblock Pllay: Efficient topological layer based on persistence landscapes.
\newblock In {\em Proceedings of the 34th International Conference on Neural Information Processing Systems}, NIPS'20, Red Hook, NY, USA, 2020. Curran Associates Inc.

\bibitem{li2022selecting}
Lei Li, Linda Yu-Ling Lan, Lei Huang, Congting Ye, Jorge Andrade, and Patrick~C Wilson.
\newblock Selecting representative samples from complex biological datasets using k-medoids clustering.
\newblock {\em Frontiers in Genetics}, 13:954024, 2022.

\bibitem{macqueen1967some}
J.~B. MacQueen.
\newblock Some methods for classification and analysis of multivariate observations.
\newblock In {\em Proceedings of 5th Berkeley Symposium on Mathematical Statistics and Probability}, volume~1, pages 281--297. University of California Press, 1967.
\newblock Retrieved 2009-04-07.
\newblock URL: \url{https://projecteuclid.org/euclid.bsmsp/1200512992}.

\bibitem{naitzat2020topology}
Gregory Naitzat, Andrey Zhitnikov, and Lek-Heng Lim.
\newblock Topology of deep neural networks.
\newblock {\em The Journal of Machine Learning Research}, 21(1):7503--7542, 2020.

\bibitem{natarajan2024morse}
Abhinav Natarajan, Thomas Chaplin, Adam Brown, and Maria-Jose Jimenez.
\newblock Morse theory for chromatic delaunay triangulations.
\newblock {\em arXiv preprint arXiv:2405.19303}, 2024.

\bibitem{olah2014manifolds}
Chris Olah.
\newblock {Neural Networks, Manifolds, and Topology}, 2014.
\newblock Accessed: 2024-06-28.
\newblock URL: \url{http://colah.github.io/posts/2014-03-NN-Manifolds-Topology/}.

\bibitem{Rathore2019}
Archit Rathore, Nithin Chalapathi, Sourabh Palande, and Bei Wang.
\newblock {TopoAct}: {Exploring} the {Shape} of {Activations} in {Deep} {Learning}.
\newblock {\em arXiv: 1912.06332}, pages 1--16, 2019.
\newblock URL: \url{http://arxiv.org/abs/1912.06332}.

\bibitem{reani2022cycle}
Yohai Reani and Omer Bobrowski.
\newblock Cycle registration in persistent homology with applications in topological bootstrap.
\newblock {\em IEEE Transactions on Pattern Analysis and Machine Intelligence}, 45(5):5579--5593, 2022.

\bibitem{rieck2018neural}
Bastian Rieck, Matteo Togninalli, Christian Bock, Michael Moor, Max Horn, Thomas Gumbsch, and Karsten Borgwardt.
\newblock Neural persistence: A complexity measure for deep neural networks using algebraic topology.
\newblock In {\em International Conference on Learning Representations}, 2018.

\bibitem{stucki2023topologically}
Nico Stucki, Johannes~C Paetzold, Suprosanna Shit, Bjoern Menze, and Ulrich Bauer.
\newblock Topologically faithful image segmentation via induced matching of persistence barcodes.
\newblock In {\em International Conference on Machine Learning}, pages 32698--32727. PMLR, 2023.

\bibitem{tolstikhin2021mlp}
Ilya~O Tolstikhin, Neil Houlsby, Alexander Kolesnikov, Lucas Beyer, Xiaohua Zhai, Thomas Unterthiner, Jessica Yung, Andreas Steiner, Daniel Keysers, Jakob Uszkoreit, et~al.
\newblock {MLP}-mixer: An all-{MLP} architecture for vision.
\newblock {\em Advances in neural information processing systems}, 34:24261--24272, 2021.

\bibitem{gonzalez2023topological}
{\'A}lvaro Torras-Casas, Eduardo Paluzo-Hidalgo, and Roc{\'\i}o Gonz{\'a}lez-D{\'\i}az.
\newblock Topological data quality via 0-dimensional persistence matchings.
\newblock {\em arXiv preprint arXiv:2306.02411}, 2023.

\bibitem{wang2020topogan}
Fan Wang, Huidong Liu, Dimitris Samaras, and Chao Chen.
\newblock {TopoGAN} : {A} {Topology}-{Aware} {Generative} {Adversarial} {Network}.
\newblock In {\em European Conference on Computer Vision}, pages 118--136. Springer, 2020.

\bibitem{watanabe_topological_2022}
Satoru Watanabe and Hayato Yamana.
\newblock Topological measurement of deep neural networks using persistent homology.
\newblock {\em Annals of Mathematics and Artificial Intelligence}, 90(1):75--92, January 2022.
\newblock URL: \url{https://link.springer.com/10.1007/s10472-021-09761-3}, \href {https://doi.org/10.1007/s10472-021-09761-3} {\path{doi:10.1007/s10472-021-09761-3}}.

\bibitem{wheeler2021activation}
Matthew Wheeler, Jose Bouza, and Peter Bubenik.
\newblock Activation landscapes as a topological summary of neural network performance.
\newblock In {\em 2021 IEEE International Conference on Big Data (Big Data)}, pages 3865--3870. IEEE, 2021.

\bibitem{yan2021link}
Zuoyu Yan, Tengfei Ma, Liangcai Gao, Zhi Tang, and Chao Chen.
\newblock Link {Prediction} with {Persistent} {Homology}: {An} {Interactive} {View}.
\newblock In {\em International conference on machine learning}, pages 11659--11669. PMLR, 2021.

\bibitem{yu2010improved}
Kai Yu, Tong Zhang, and Yihong Gong.
\newblock Improved local coordinate coding using local tangents.
\newblock In {\em Proceedings of the 27th International Conference on Machine Learning (ICML)}, pages 1215--1222, 2010.
\newblock Table~1 reports a 12.0\% error rate for a linear SVM trained directly on raw MNIST images.

\bibitem{zheng2021topological}
Songzhu Zheng, Yikai Zhang, Hubert Wagner, Mayank Goswami, and Chao Chen.
\newblock Topological detection of trojaned neural networks.
\newblock {\em Advances in Neural Information Processing Systems}, 34:17258--17272, 2021.

\bibitem{zhou2020evaluating}
Sharon Zhou, Eric Zelikman, Fred Lu, Andrew~Y Ng, Gunnar~E Carlsson, and Stefano Ermon.
\newblock Evaluating the disentanglement of deep generative models through manifold topology.
\newblock In {\em International Conference on Learning Representations}, 2020.

\bibitem{zia2023topological}
Ali Zia, Abdelwahed Khamis, James Nichols, Zeeshan Hayder, Vivien Rolland, and Lars Petersson.
\newblock Topological {Deep} {Learning}: {A} {Review} of an {Emerging} {Paradigm}, February 2023.
\newblock arXiv:2302.03836 [cs].
\newblock URL: \url{http://arxiv.org/abs/2302.03836}.

\end{thebibliography}

\appendix

\section{Preliminaries}
\label{app:prelim}
We recall basic concepts from algebraic and computational topology.
We will work with filtrations of simplicial complexes, chain complexes, persistence modules and their decompositions.

For a simplicial complex $K$, and $k \geq 0$, let $C_k(K)$ denote the $\mathbb{Z}_2$-vector space whose basis is the set of $k$-simplices in $K$.
Let $d_k:C_k(K) \to C_{k-1}(K)$ denote the differential, where $C_{-1}(K)=0$.
$(C_k(K),d_k)_{k \geq 0}$ is called the simplicial chain complex of $K$ and denoted $C(K)$ for short.
Let $Z_k(K) = \ker(d_k)$, $B_k(K) = \im(d_{k+1})$ and $H_k(K) = \ker(d_k)/\im(d_{k+1})$.
$H_k(K)$ is called the simplicial homology of $K$ in degree $k$.
An inclusion of simplicial complexes $i:L \hookrightarrow K$ induces a chain map $C(i): C(L) \to C(K)$ consisting of linear maps $C_k(i): C_k(L) \to C_k(K)$ for each $k \geq 0$ such that for all $k$, $d_k \circ C_k(i) = C_{k-1}(i) \circ d_k$, where $C_{-1}(i)=0$.
It follows that for each $k \geq 0$ there is an induced map $H_k(i): H_k(L) \to H_k(K)$.

A \new{persistence module} $M$ consists of $\mathbb{Z}_2$-vector spaces $M_i$ for $1 \leq i \leq n$ and linear maps $M_{i \leq j}:M_i \to M_j$ for $1 \leq i \leq j \leq n$ such for each $i$, $M_{i \leq i}$ is the identity map and for $1 \leq i \leq j \leq k \leq n$, $M_{j \leq k} \circ M_{i\leq j} = M_{i \leq k}$.
A map of persistence modules $\varphi: M \to N$ consists of linear maps $\varphi_i: M_i \to N_i$ for $1 \leq i \leq n$ such that $N_{i \leq j} \circ \varphi_i = \varphi_j \circ M_{i \leq j}$ for all $1 \leq i \leq j \leq n$.
The category of persistence modules is an abelian category, so we have direct sums of persistence modules (defined elementwise) and maps of persistence modules have images, kernels, and cokernels (also defined elementwise). 

Given an interval $I$ in $(\{1,\ldots, n\},\leq)$, let $\chi_I$ denote the persistence module given by $(\chi_I)_k = \mathbb{Z}_2$ if $k \in I$ and $0$ otherwise and the linear maps in the persistence module are identity maps wherever possible. $\chi_I$ is called the interval module supported on the interval $I$.
The \new{barcode of a persistence module} $M$ is a collection of intervals or \new{bars} 
$\{I_j\}_{j=1}^N$ such that $M \cong \oplus_{j=1}^N \chi_{I_j}$.
The barcode is unique up to reordering. 
We sometimes denote the bar $[i,j]$ by $[i,j+1)$.

\section{More detailed definitions}
\label{app:details}
For inclusion of our filtered simplicial complexes $\iota: L \hookrightarrow K$  and each $k \geq 0$ we have persistence modules $C_k(L)$ and $C_k(K)$ and
induced maps of persistence modules $d_k: C_k(L) \to C_{k-1}(L)$, $d_k:C_k(K) \to C_{k-1}(K)$ and $C_k(\iota):C_k(L) \hookrightarrow C_k(K)$.
For each $k \geq 0$ we also have persistence modules $Z_k(L)$, $B_k(L)$, $Z_k(K)$, $B_k(K)$, $H_k(L)$ and $H_k(K)$ and induced maps of persistence modules $Z_k(\iota): Z_k(L) \to Z_k(K)$, $B_k(\iota): B_k(L) \to B_k(K)$ and $H_k(\iota): H_k(L) \to H_k(K)$. The latter is given by the
following commutative diagram of vector spaces and linear maps.

The persistent homology vector spaces~\cite{edelsbrunner2010computational} of the persistence module $H_k(L)$ are given by $\im ((H_k(L))_{i \leq j}) = \frac{Z_k(L_i)}{B_k(L_j) \cap Z_k(L_i)}$, for $1 \leq i \leq j \leq n$.
The image persistent homology vector spaces~\cite{cohen2009persistent} are the persistent homology vector spaces of the persistence module given by the image of $H_k(\iota)$.
For $1 \leq i \leq j \leq n$,
$\im ( (\im (H_k(\iota)))_{i \leq j})
= \frac{Z_k(L_i)}{B_k(K_j) \cap Z_k(L_i)}$, 
noting that one of these images comes from 
the definition of persistent homology,
the other from the fact that we consider 
image persistence.


\section{Proof of algorithm correctness}
\label{app:proof}
We prove the correctness of the main algorithm.

\begin{theorem}[Mixup Decomposition Theorem]
Algorithm~\ref{alg:mixup-decomposition} returns a correct 
mixup barcode of $L \hookrightarrow K$.
\end{theorem}

\hw{Could someone please carefully check the below proof?}
\begin{proof}


The matrix $BL$ is a submatrix of $BK$ so that the common nonzero columns correspond to the same simplices in $L$. 
Additionally, reordering the rows in $BK$ sets it up for image persistence computations, as explained above. Therefore, reducing $BL$ and $BK$ using REDUCE yields the  persistent homology of $L$ and the image persistence of $L \hookrightarrow K$ respectively. 
Moreover, the indexing of simplices in the two 
matrices is consistent and allows us to match 
the birth and death simplices, as explained next.

Specifically, each index of a $k$-simplex $\sigma$ with $BL[\sigma]=0$ corresponds to a unique homology class $\gamma$ in $H_k(L)$, namely the class \emph{created} by the simplex $\sigma$.  (We recall that $BL[\sigma]=0$ implies that $\sigma$ gives birth to a homology class in degree $k$~\cite{edelsbrunner2000topological}.)
Whenever $\tau$ and $\tau'$ exist, each is unique since each reduced matrix is guaranteed to have unique pivots~\cite{edelsbrunner2010computational}. In particular, they do not depend on the variant of the reduction algorithm, order of column additions, etc. If $\tau$ or $\tau'$ does not exist, 
in the case of homology classes persisting forever, we handle it in a straightforward way. Note that if $\tau$ exists, so does $\tau' = \tau$.

By construction, the pair $(\sigma, \tau)$ is an index persistence pair arising from $L$ and $\tau$ is the simplex that destroys $\gamma$. As for $(\sigma, \tau')$ there are two options: (1) it is the index \emph{image} persistence pair arising from $L \hookrightarrow K$~\cite{cohen2009persistent} and $\tau'$ is the index of the simplex that gives premature death to the image of $\gamma$; or (2) it coincides with $(\sigma, \tau)$ in which case no premature death occurs.

Therefore each triple $(\sigma, \tau{'}, \tau)$ encodes the decomposition of the index persistence bar into the index image persistence 
sub-bar $(\sigma, \tau{'})$ and sub-bar ($\tau{'}, \tau)$, which we call an index mixup sub-bar. 
\end{proof}

\section{Simple examples}

\begin{figure}
    \centering
    \includegraphics[width=1\linewidth]{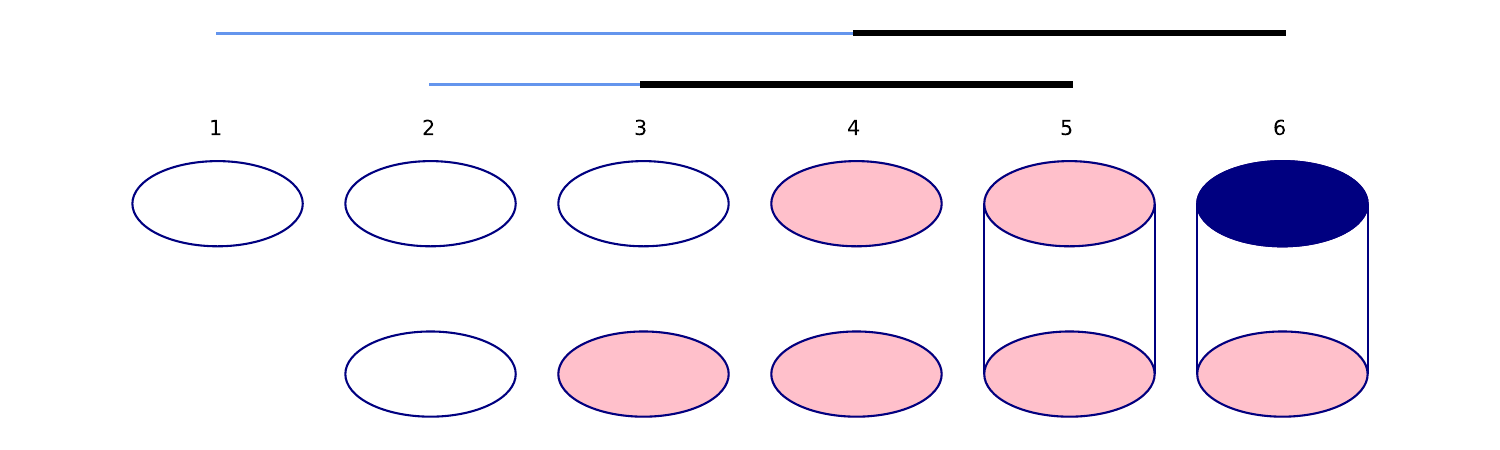}
    \caption{  
    In this example $K$ is a cylinder capped with disks. The dark cells (two circles, one disk and the cylinder) come from $L$, and the light cells (the two disks) come from $K \setminus L$. The new disk at $L_6$ is meant as another cell bounding the circle, distinct 
    from the disk at $K_4$ (consistent with our assumptions).    
    The mixup triples of $L \hookrightarrow K$ in degree $1$ are $((1,4,6), (2,3,5))$, as illustrated with the mixup barcode plotted above.
    In particular, the mixup sub-bars $([4,6), [3,5))$ quantify the shortening of the lifetime of the two cycles. The kernel persistence bars are $([3,6), [4,5))$ and do not correspond to the shortening of the lifetimes of the cycles.
    }
    \label{fig:ex1}
\end{figure}

\begin{figure}
    \centering
    \includegraphics[width=1\linewidth]{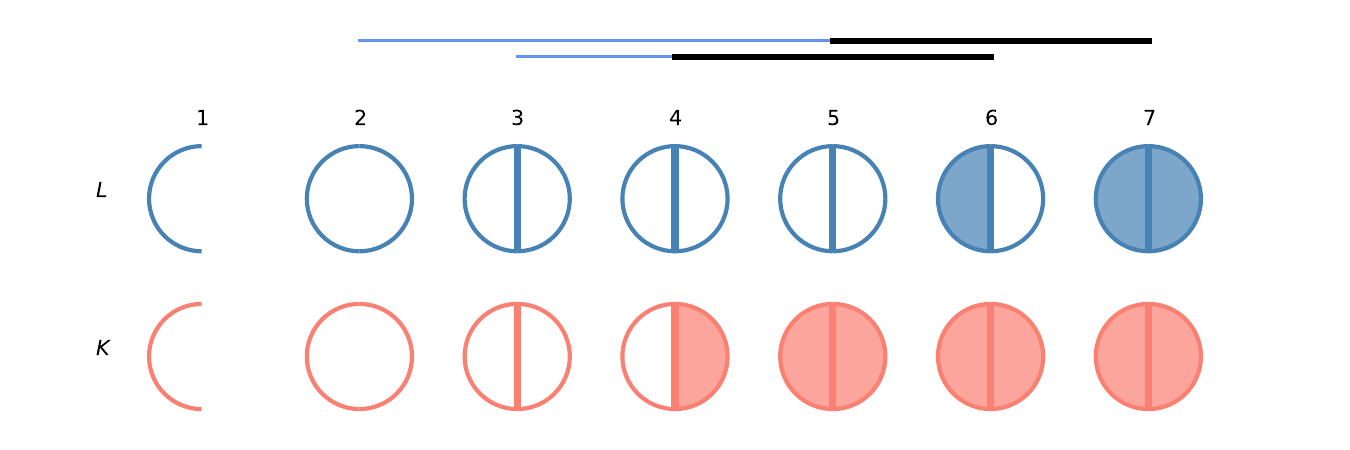}
    \caption{This example is presented differently. The top filtration is $L$
    and the bottom filtration is $K$. The $1$-cells are half-circles and a segment, and the $2$-cells are half-disks.
    We assume that the cells added to $L_6$ and $L_7$ are distinct from the cells
    added in $K_5$ and $K_4$, respectively. We consider the mixup barcode of $L \hookrightarrow K$ in degree $1$.
    }
    \label{fig:ex2}
\end{figure}

\label{app:examples}
Figure~\ref{fig:ex1} shows a simple example (provided by a reviewer) using cell complexes. In particular, it illustrates 
that the mixup sub-bars capture the shortening of the lifetimes of homology classes. It also shows that the mixup sub-bars do not generally coincide with the kernel persistence bars. In such cases the kernel persistence bars do not have a clear interpretation, while the mixup bars do.


Figure~\ref{fig:ex2} shows the behavior of a mixup barcode when cycle representatives of the generators of the persistence module $H_1(L)$ and $\im H_1(\iota)$ are different. Namely if we label 
the cells arriving at $L_1$, $L_2$, $L_3$ as $a,b,c$, then they are the cellular chains $\{a+b, a+c\}$
and $\{a+b, b+c\}$, respectively. 
The premature death of the homology class of $a+c$ corresponds to the death of the class of $b+c$. We remark that these two classes share the birth simplex $c$.

\end{document}